\RequirePackage{etex}
\documentclass{amsart}

\usepackage[nocompress,noadjust]{cite}

\usepackage{tikz}
\usepackage{tikz-cd}
\usetikzlibrary{cd}
\usetikzlibrary{positioning, arrows.meta, matrix}

\usepackage{fancybox}

\usepackage{listings}
\usepackage{docmute}
\usepackage{amssymb, enumerate, calligra}
\usepackage{amsthm}
\usepackage{amsmath}
\usepackage{amsxtra}
\usepackage{latexsym}
\usepackage{mathrsfs}
\usepackage[all,cmtip]{xy}
\usepackage[all]{xy}
\usepackage{enumitem}
\usepackage{xcolor}
\usepackage{graphicx}
\usepackage{comment}
\usepackage{mathabx,epsfig}
\usepackage{stmaryrd}
\def\acts{\ \rotatebox[origin=c]{-90}{$\circlearrowright$}\ }
\def\racts{\ \rotatebox[origin=c]{90}{$\circlearrowleft$}\ }

\def\uacts{\ \rotatebox[origin=c]{180}{$\circlearrowright$}\ }

\usepackage[hypertexnames=false]{hyperref}
\usepackage[nameinlink]{cleveref}
\usepackage{autonum}

\usepackage{mathtools}

\newcommand{\Yohsuke}[1]{{\color{blue}(Yohsuke: #1)}}

\makeatletter
\renewcommand\th@plain{\slshape}
\makeatother
\newtheoremstyle{plain}
 {2mm}
 {2mm}
 {\slshape}
 {}
 {\bfseries}
 {.}
 {.5em}
 {}
\theoremstyle{plain}
\newtheorem{theorem}{Theorem}[section]
\newtheorem{corollary}[theorem]{Corollary}
\newtheorem{lemma}[theorem]{Lemma}
\newtheorem{proposition}[theorem]{Proposition}
\newtheorem{claim}[theorem]{Claim}
\newtheorem*{claim*}{Claim}
\newtheorem{conjecture}[theorem]{Conjecture}
\newtheorem{question}[theorem]{Question}

\newtheorem{mainthm}{Theorem}

\newtheorem{mainprop}[mainthm]{Proposition}

\newtheoremstyle{definition}
 {2mm}
 {2mm}
 {\normalfont}
 {}
 {\bfseries}
 {.}
 {.5em}
 {}
\theoremstyle{definition}

\newtheorem{definition}[theorem]{Definition}

\newtheorem{remark}[theorem]{Remark}
\newtheorem{example}[theorem]{Example}
\newtheorem*{acknowledgements}{Acknowledgements}

\crefname{section}{Section}{Sections}

\crefname{theorem}{Theorem}{Theorems}
\crefname{corollary}{Corollary}{Corollaries}
\crefname{lemma}{Lemma}{Lemmas}
\crefname{lemma}{Lemma}{Lemmas}
\crefname{proposition}{Proposition}{Propositions}
\crefname{claim}{Claim}{Claims}
\crefname{definition}{Definition}{Definitions}
\crefname{notation}{Notation}{Notations}
\crefname{problem}{Problem}{Problems}
\crefname{note}{Note}{Notes}
\crefname{remark}{Remark}{Remarks}
\crefname{example}{Example}{Examples}
\crefname{conjecture}{Conjecture}{Conjectures}
\crefname{question}{Question}{Questions}
\crefname{mainthm}{Theorem}{Theorems}
\crefname{mainprop}{Proposition}{Propositions}

\crefname{enumi}{}{}
\crefname{enumii}{}{}
\crefname{enumiii}{}{}

\crefformat{equation}{{\upshape(#2#1#3)}}

\numberwithin{equation}{section}

\def\C{{\mathbb C}}
\def\Q{{\mathbb Q}}
\def\R{{\mathbb R}}
\def\Z{{\mathbb Z}}
\def\P{{\mathbb P}}

\def\F{{\mathbb F}}

\def\Hom{\mathop{\mathrm{Hom}}\nolimits}

\def\p{{ \mathfrak{p}}}     
\def\q{{ \mathfrak{q}}}     
\def\O{{ \mathcal{O}}}
 
\def\a{{ \mathfrak{a}}}  
  
\def\I{{ \mathcal{I}}}

\def\M{{ \mathcal{M}}}
\def\Y{{ \mathcal{Y}}}

\def\lb{\llbracket}
\def\rb{\rrbracket}

\newcommand{\wedgeop}[1]{\scalebox{1}[1.5]{$\wedge$}^{#1}\,}

\renewcommand{\mod}[1]{(\mathrm{mod}\ #1)}

\DeclareMathOperator{\pr}{pr}

\DeclareMathOperator{\id}{id}

\DeclareMathOperator{\End}{End}
\DeclareMathOperator{\Spec}{Spec}

\DeclareMathOperator{\ord}{ord}

\DeclareMathOperator{\Fix}{Fix} 
\DeclareMathOperator{\Sing}{Sing} 
\DeclareMathOperator{\Frac}{Frac}
\DeclareMathOperator{\Stab}{Stab}

\renewcommand{\a}{\alpha}

\renewcommand{\b}{\beta}

\newcommand{\g}{\gamma}
\renewcommand{\d}{\delta}

\newcommand{\f}{\varphi}

\renewcommand{\l}{\lambda}
\renewcommand{\k}{\kappa}

\newcommand{\s}{\sigma}

\newcommand{\bD}{\overline{\mathbf{D}}}

\def\X{{ \mathcal{X}}}
\def\Y{{ \mathcal{Y}}}

\newcommand{\hh}{\hat{h}}

\allowdisplaybreaks

\title[Preimages Question]
{On Preimages Question}
\author{Yohsuke Matsuzawa}
\author{Kaoru Sano}

\address{Department of Mathematics, Graduate School of Science, Osaka Metropolitan University, 3-3-138, Sugimoto, Sumiyoshi, Osaka, 558-8585, Japan}
\address{NTT Institute for Fundamental Mathematics, NTT Communication Science Laboratories, NTT Corporation, 2-4, Hikaridai, Seika-cho, Soraku-gun, Kyoto 619-0237, Japan}

\email{matsuzaway@omu.ac.jp}
\email{kaoru.sano@ntt.com}

\begin{document}

\begin{abstract}
    For a surjective self-morphism on a projective variety defined over a number field,
    we study the preimages question, which asks if the set of rational points
    on the iterated preimages of an invariant closed subscheme eventually stabilize.
    We prove the answer is positive for self-morphisms on $\P^1 \times \P^1$.
    We also prove that the answer is positive even if we allow finite extension of the number field by
    a bounded degree if the morphism is \'etale or the invariant subscheme is $0$-dimensional.
    We prove certain uniformity of the stabilization for self-morphisms on abelian varieties.
    When the subscheme is not invariant, we provide counterexamples and
    a result in a positive direction for products of polynomial maps on $\P^1 \times \P^1$.
\end{abstract}


\maketitle

\tableofcontents

\section{Introduction}



Let $f \colon X \longrightarrow X$ be a surjective morphism on a projective 
variety $X$ defined over a number field $K$.
For a closed subscheme $Y \subset X$, the iterated preimages $f^{-n}(Y)$
tend to become complicated as $n$ approaches $\infty$, and so
it is reasonable to expect that the $K$-rational points on $f^{-n}(Y)$
might be sparse in a certain sense.
Such a kind of question is investigated in \cite{BMS23},
motivated by a question proposed in \cite{MMSZ23}.
The central problem is the following (\cite[Question 8.4(1)]{MMSZ23},\cite[Question 1.1]{BMS23}).

\begin{question}[Preimages Question (PIQ)]\label{que:PIQ}
    Let $K$ be a number field.
    Let $X$ be a projective variety over $K$ 
    (i.e.\ integral projective scheme over $K$) and 
    $f \colon X \longrightarrow X$ a surjective morphism over $K$.
    Let $Y \subset X$ be an $f$-invariant closed subscheme 
    (i.e.\ $f|_Y$ factors through $Y$).
    Then is there $s_0 \in \Z_{\geq 0}$ such that for all $s \geq s_0$, 
    we have
    \begin{align}
        \left(f^{-s-1}(Y) \smallsetminus f^{-s}(Y) \right)(K) = \emptyset\ ?
    \end{align}
\end{question}

Let us give some comments on the assumptions.
We cannot remove the projectivity of $X$ as there is a counterexample 
\cite[section 6.1]{BMS23}.
The $f$-invariance of $Y$ is more mysterious.
We can pose a similar question for $Y$ that is not necessarily $f$-invariant (see \cref{que:PIQ_non_invariant}). While there are counterexamples (cf. \cref{sec:Preimages Question when the subvariety is not invariant}), it seems plausible to assume that certain additional conditions on $f$ or $Y$ might yield a positive answer. The $f$-invariance of $Y$ can be seen as one such condition, but we anticipate that other types of conditions exist. Currently, we do not comprehend what causes the set of rational points on $f^{-n}(Y)$ to be sparse.

Let us also give a comment on other ground fields.
Obviously, the question is nonsense if $K$ is algebraically closed.
A slightly non-trivial fact is that it has a counterexample 
over $p$-adic fields in general (see \cite[section 6.2]{BMS23}).
However, for a specific type of morphisms, PIQ is true 
over $p$-adic fields or finitely generated fields over $\Q$
(cf.\ \cite[Proof of Theorem 3.1 Step 2]{BMS23}, \cref{thm:piq_P1P1}).

\begin{remark}
    The statement of \cref{que:PIQ} makes sense for arbitrary self-map $f \colon X \longrightarrow X$
    on a set $X$ and its $f$-invariant subset $Y \subset X$.
    We use the terminology Preimages Question or PIQ for such situations flexibly.
    For example, PIQ is true for $(X,f,Y)$ means there is $s_0 \in \Z_{\geq 0}$ such that for all $s \geq s$, we have
    \begin{align}
        f^{-s-1}(Y) \smallsetminus f^{-s}(Y) = \emptyset.
    \end{align}
    In this terminology, \cref{que:PIQ} is equivalent to PIQ for $f \colon X(K) \longrightarrow X(K)$ and subset $Y(K) \subset X(K)$.
\end{remark}

\begin{remark}
    For a map $f \colon X \longrightarrow X$ on a set $X$ and 
    an $f$-invariant subset $Y \subset X$, the following are equivalent.
    \begin{enumerate}
        \item PIQ is true for $(X, f, Y)$.
        \item PIQ is true for $(X, f^n, Y)$ for some $n \geq 1$.
        \item There is $s_0 \geq 0$ with the following property: for every $x \in X$, if $f^s(x) \in Y$ for some $s \geq 0$, then $f^{s_0}(x) \in Y$.
    \end{enumerate}
\end{remark}

Back to \cref{que:PIQ}, the original \cite[Question 8.4(1)]{MMSZ23} was
the following slightly stronger question.

\begin{question}\label{que:PIQ_fin_ext}
    Let $K$ be a number field.
    Let $X$ be a projective variety over $K$ and 
    $f \colon X \longrightarrow X$ a surjective morphism over $K$.
    Let $Y \subset X$ be an $f$-invariant closed subscheme.
    Let $d \in \Z_{\geq 0}$.
    Then is there a $s_0 \in \Z_{\geq 0}$ such that for all $s \geq s_0$, 
    we have
    \begin{align}
        \left(f^{-s-1}(Y) \smallsetminus f^{-s}(Y) \right)(L) = \emptyset
    \end{align}
    for all field extensions $K \subset L$ with $[L:K] \leq d$.
\end{question}

It turns out:

\begin{mainprop}[\cref{prop:twoPIQequiv}]\label{prop:main_equiv_two_PIQ}
    \cref{que:PIQ} for all $K, X, f$, and $Y$ is equivalent to 
    \cref{que:PIQ_fin_ext} for all $K, X, f, Y$, and $d$.
\end{mainprop}

We verify \cref{que:PIQ_fin_ext} for the following cases.

\begin{mainthm}[\cref{thm:PIQ-etale,prop:PIQ_fixedpoint}]\label{thm:main_fixedpoint_and_etale}
    The answer to \cref{que:PIQ_fin_ext} is affirmative if
    either $\dim Y = 0$, or $f$ is \'etale.
\end{mainthm}

Note that \cref{que:PIQ} was proven for \'etale morphisms 
in \cite{BMS23}. 

In \cite{BMS23}, the authors proved that \cref{que:PIQ} is true for
self-morphisms on $\P^1 \times \P^1$ of the form $f \times f$
and the diagonal subvariety.
We generalized this theorem to arbitrary self-morphisms on 
$\P^1 \times \P^1$ and any invariant closed subscheme
(as well as to finitely generated fields).

\begin{mainthm}[\cref{thm:piq_P1P1}]\label{thm:main_piq_P1P1}
    Let $K$ be a finitely generated field over $\Q$. 
    Let $\f \colon \P^1_K \times_K \P^1_K \longrightarrow \P^1_K \times_K \P^1_K$
    be a surjective morphism over $K$.
    Let $Y \subset \P^1_K \times_K \P^1_K$ be a $\f$-invariant closed subscheme.
    Then there is $s_0 \geq 0$ such that for all $s \geq s_0$, we have
    \begin{align}
        \left(\f^{-s-1}(Y) \smallsetminus \f^{-s}(Y) \right)(K) = \emptyset.
    \end{align}
\end{mainthm}

We have mentioned that the invariance of $Y$ in \cref{que:PIQ}
cannot be removed in general.
But we have the following particular case that the answer is positive.

\begin{mainthm}[\cref{thm:GPIQ_P1P1}]\label{thm:main_gpiq_P1P1}
    Let $K$ be a number field.
    Let $f,g \colon \P^1_K \longrightarrow \P^1_K$ be polynomial maps of the same degree $d \geq 2$.
    Consider the product morphism 
    $f \times g \colon \P^1_K \times \P^1_K \longrightarrow \P^1_K \times \P^1_K$
    and the diagonal subvariety $\Delta \subset \P^1_K \times \P^1_K$. 
    There is $s_0\in \Z_{\geq 0}$ such that for all $s \geq s_0$, we have
    \begin{align}
        \biggl((f\times g)^{-s-1}(\Delta) \smallsetminus \bigcup_{0 \leq i \leq s} (f \times g)^{-i}(\Delta) \biggr)(K) = \emptyset.
    \end{align}
\end{mainthm}

We have seen that the answer to \cref{que:PIQ} is positive 
for some classes of morphisms and varieties.
It is then natural to ask how much extent the (minimum of the) number $s_0$ in \cref{que:PIQ} depends on $f$, $Y$, and $X$.
When $X$ is an abelian variety, we prove the following uniformity on $s_0$.

\begin{mainthm}[\cref{thm:uniform-PIQ-ab-var}]\label{thm:main_uniformity_ab_var}
    Let $X$ be an abelian variety over a number field $K$.
    Then there exists $s_0 \in \Z_{\geq 0}$ depending only on $X$ with the 
    following property.
    For any surjective morphism $f \colon X \longrightarrow X$ and 
    any reduced closed subscheme $Y \subset X$ with $f(Y) \subset Y$, we have 
    \begin{align}
        (f^{-s_0 -1}(Y) \smallsetminus f^{-s_0}(Y))(K) = \emptyset.
    \end{align}
\end{mainthm}

\begin{remark}
    In the course of the proof of \cref{thm:main_uniformity_ab_var},
    we prove a uniform bound of the periods of periodic abelian subvarieties 
    under isogenies \cref{lem:period-of-absubvar}.
    While our proof heavily relies on several special properties of abelian varieties,
    it is worth noting the similarity of this and 
    the problem raised in \cite[Question 8.4 (2)]{MMSZ23}.
\end{remark}

\begin{remark}
    Under the assumption of general uniform boundedness conjecture on the torsion points 
    of abelian varieties, we obtain a stronger version of \cref{thm:main_uniformity_ab_var},
    subject to the condition that $Y$ is geometrically irreducible.
    Refer to \cref{prop:uniform_PIQ_tors_conj} for details.
\end{remark}

\noindent
{\bf Notation.}

\begin{enumerate}
    \item A {\it variety} over a field $k$ is an irreducible reduced separated scheme of finite type over $k$.
    For a separated scheme $X$ of finite type over $k$,
    a {\it subvariety} of $X$ is an irreducible reduced closed subscheme of $X$.
    
    \item For a self-map $f \colon X \longrightarrow X$ on a set $X$,
    $\Fix(f, X)$ is the set of fixed point, i.e. 
    $\Fix (f, X) = \{x \in X \mid f(x) = x\}$.
    
    \item For a rational number $a$, 
    we set $|a|_p = p^{-v_p(a)}$, where $v_p$ is the $p$-adic valuation.
    In particular, $|p^n|_p = p^{-n}$ for $n \in \Z$.
    
    \item 
    A valued field $(k, |\ |)$ is a field equipped with 
    a multiplicative absolute value $|\ |$.

    \item For a non-archimedean valued field $k$,
    the classical (closed) disc of radius $r > 0$ is denoted by $\bD_k(r)$:
    $\bD_k(r) = \{a \in k \mid |a|\leq r\}$.
\end{enumerate}

\noindent
{\bf Organization of the paper.}

In \cref{sec:Equivalence of two Preimages Questions}, 
we prove \cref{prop:main_equiv_two_PIQ}, i.e., the two preimages questions are equivalent.
In \cref{sec:Preimages Question for a fixed point}, 
we prove that \cref{que:PIQ_fin_ext} is true if the invariant subscheme is $0$-dimensional.
In \cref{sec:Etale morphisms}, we prove it for \'etale morphisms.
In \cref{sec:Preimages Question when the subvariety is not invariant},
we introduce a variant of \cref{que:PIQ} where the subscheme is not assumed to be invariant.
We prove \cref{thm:main_gpiq_P1P1}, which gives an affirmative answer.
We also note that a certain dynamical structure produces counterexamples
to the generalized question.
In \cref{sec:uniformity of s0 ab var}, we prove \cref{thm:main_uniformity_ab_var}.
In \cref{sec:Preliminaries on polydisc algebras,sec:padic uniformization},
we recall some facts and prepare several lemmas that are used in the proof of
\cref{thm:main_piq_P1P1}.
In \cref{sec:piq P1P1}, we prove \cref{thm:main_piq_P1P1}.

\begin{acknowledgements}
 The authors would like to thank Yusuke Okuyama, Joseph Silverman, 
and Yu Yasufuku for helpful discussions.
The first author is supported by JSPS KAKENHI Grant Number JP22K13903.
The second author is supported by JSPS KAKENHI Grant Number JP20K14300.
\end{acknowledgements}

\section{Equivalence of two Preimages Questions}\label{sec:Equivalence of two Preimages Questions}

In this section, we prove the following.

\begin{proposition}\label{prop:twoPIQequiv}
    \cref{que:PIQ} for all $K, X, f$, and $Y$ is equivalent to 
    \cref{que:PIQ_fin_ext} for all $K, X, f, Y$, and $d$.
\end{proposition}

\begin{remark}
    We do not claim the equivalence of two questions for individual dynamical systems $(X,f, Y)$.
\end{remark}

\begin{proof}
    It is clear that \cref{que:PIQ_fin_ext} implies \cref{que:PIQ}.
    Assume \cref{que:PIQ} is true for all $K$, $X$, $f$, and $Y$.
    Let $K$, $X$, $f$, $Y$, and $d$ as in \cref{que:PIQ_fin_ext}.
    By replacing $K$ with a finite extension and replacing $f$ with iterates,
    we may assume $X$ is geometrically irreducible.
    It is enough to show the statement for $L$ that is Galois over $K$ with $[L: K] = d$.
    Now let $X^{(d)} := X^d / \mathfrak{S}_d$ be the symmetric product of $X$.
    Note that $X^{(d)}$ is projective over $K$ (cf.\ \cref{rem:fingrpqt_is_proj}).

    Let us consider the following commutative diagram
    \begin{equation}
        \begin{tikzcd}
            Y^{d} \arrow[r,hookrightarrow] \arrow[d] & X^d \arrow[r, "f \times \cdots \times f"] \arrow[d,"\pi"] & X^d \arrow[d,"\pi"]  \\
            Y^{(d)} \arrow[r] \arrow[rd,twoheadrightarrow] & X^{(d)} \arrow[r, "g"]& X^{(d)}\\[-1em]
            & Z \arrow[u, phantom, "\subset",sloped] \arrow[r] & Z \arrow[u, phantom, "\subset",sloped]
        \end{tikzcd}
    \end{equation}
    where $\pi$ is the quotient morphism, $Z$ is the scheme theoretic image of $Y^{(d)} \longrightarrow X^{(d)}$,
    and $g$ is the induced surjective morphism.
    Let $L$ be any Galois extension of $K$ of degree $d$ with Galois group
    $G = \{\s_1, \dots ,\s_d\}$.
    For any point $x \in X(L)$, we define $\widetilde{x} \in X^d(L)$ by
    \begin{equation}
        \begin{tikzcd}
            \Spec L \arrow[r,"\widetilde{x}"] \arrow[d,"\s_i^*"] & X^d \arrow[d,"\pr_i"] \\
            \Spec L \arrow[r,"x"]& X
        \end{tikzcd}
    \end{equation}
    for all $i=1,\dots,d$.
    Then we have $\pi \circ \widetilde{x}$ factors through $\Spec K$ and 
    defines a $K$-point $\overline{x} \in X^{(d)}(K)$.
    It is easy to see that
    \begin{align}
        x \in Y(L) \iff \overline{x} \in Z(K).
    \end{align}

    Now let $s_0 \geq 0$ be as in \cref{que:PIQ} for $(X^{(d)}, g, Z)$.
    If $x \in X(L)$ satisfies $f^s(x) \in Y(L)$ for some $s \geq 0$, 
    then $\overline{f^s(x)} \in Z(K)$.
    Since $\overline{f^s(x)} = g^s(\overline{x})$,
    we have $g^{s_0}(\overline{x}) \in Z(K)$ and therefore $f^{s_0}(x) \in Y(L)$.
    Since $s_0$ depends only on $X^{(d)}, g, Z$, and $K$, we are done.
    
\end{proof}

\begin{remark}\label{rem:fingrpqt_is_proj}
    To see $X^{(d)}$ is projective, it is enough to prove that
    $X^{(d)}_{\overline{K}} = X_{\overline{K}}^d / \mathfrak{S}_d$ is projective.
    In general, finite group action on a projective variety over an algebraically closed field 
    is linearizable by a very ample line bundle (cf.\ \cite[Proposition 3.4.5]{Br18}).
    Thus, the standard construction of quotient in GIT shows that
    the quotient $X_{\overline{K}}^d / \mathfrak{S}_d$ is projective (cf.\ \cite[Proposition 8.1, Exercise 8.5]{Dol03}).
\end{remark}

\section{Preimages Question for a fixed point}\label{sec:Preimages Question for a fixed point}

The simplest case of \cref{que:PIQ} is when the invariant subset $Y$
consists of a single rational point $y \in X(K)$.
If, for example, $f$ is polarized, it is easy to see that 
the set $\bigcup_{s \geq 0}f^{-s}(y)(\overline{K})$ is a set of bounded height and
hence the answer to \cref{que:PIQ}, as well as \cref{que:PIQ_fin_ext}, is affirmative.
In general, we cannot use such a height argument, but the answer is still affirmative.

\begin{proposition}\label{prop:PIQ_fixedpoint}
    Let $K$ be a number field.
    Let $X$ be a projective variety over $K$ and $f \colon X \longrightarrow X$
    a surjective morphism.
    Let $Y \subset  X$ be a $0$-dimensional $f$-invariant closed subscheme.
    Then the answer to \cref{que:PIQ_fin_ext} is affirmative, i.e.\ 
    for $d \in \Z_{\geq 1}$, there is $s_0 \in \Z_{\geq 0}$ such that
    for all $s \geq s_0$, we have
    \begin{align}
        \left(f^{-s-1}(Y) \smallsetminus f^{-s}(Y) \right)(L) =  \emptyset
    \end{align}
    for all field extensions $K \subset L$ with $[L:K] \leq d$.
\end{proposition}

\begin{proof}
    By taking base change to a finite extension of $K$,
    we may assume $Y$ is a finite set of $K$-rational points.
    Then, replacing $f$ with an iterate, we may further assume that
    $Y = \{y\}$ where $y \in X(K)$ is a fixed point of $f$.
    
    Suppose for any $s_0 \geq 0$, there are $s \geq s_0$ and $K \subset L$ with $[L:K] \leq d$
    such that
    \begin{align}
        \left(f^{-s-1}(y) \smallsetminus f^{-s}(y) \right)(L) \neq \emptyset.
    \end{align}
    Then for all $s \geq 0$, 
    \begin{align}
        A_{s+1} := \{ z \in f^{-s-1}(y) \smallsetminus f^{-s}(y) \mid [\k(z) : K] \leq d  \} \neq \emptyset
    \end{align}
    where $\k(z)$ is the residue field of $z$.
    We set $A_0=\{y\}$.
    If $z \in A_{s+1}$, then $f(z) \in f^{-s}(y) \smallsetminus f^{-(s-1)}(y)$
    and $[\k(f(z)) : K] \leq [\k(z) : K] \leq d$.
    Thus we have $f(A_{s+1}) \subset A_s$.
    Thus we can take a $(x_s)_{s \geq 0} \in \varprojlim A_s$.
    By replacing $K$ with a finite extension, we may regard $x_s$ as $K$-rational points of $X$
    for all $s$. Note that $x_0=y$.

    Let us take a model $\X$ and $\widetilde{f}$ of $X$ and $f$ 
    over the ring of $S$-integers $\O_{K,S}$ of $K$.
    For a finite place $v \not\in S$, $\F_v$ denotes the residue field of $\O_{K,S}$ at $v$,
    and we write $X_{\F_v} = \X \times_{\O_{K,S}} \F_v$, 
    $f_{\F_v} = \widetilde{f} \times_{\O_{K,S}} \F_v$.
    Every $x_s$ corresponds to $\widetilde{x_s} \in \X(\O_{K,S})$.
    Its reduction to $\F_v$ is denoted by $x_s(v)$.
    \begin{equation}
        \begin{tikzcd}
            &[-3em] & \widetilde{f} \arrow[d,phantom,"\uacts"] & &[-3em] \\[-1em]
            f \arrow[r,phantom,"\acts"] & X \arrow[d] \arrow[r] & \X \arrow[d] & X_{\F_v} \arrow[l] \arrow[d] & f_{\F_v} \arrow[l,phantom, "\racts"]\\
            &\Spec K \arrow[r] \arrow[u,bend left=3em, "x_s",pos=0.45] & \Spec \O_{K,S} \arrow[u,bend left=3em, "\widetilde{x_s}",pos=0.45] & \Spec \F_v \arrow[l] \arrow[u,bend right=3em, "x_s(v)",swap,pos=0.5] &
        \end{tikzcd}
    \end{equation}
    Since $X_{\F_v}(\F_v)$ is a finite set and $f_{\F_v}(x_0(v))=x_0(v)$,
    there is $n_v \geq 0$ with the following property:
    any point $a \in X_{\F_v}(\F_v)$ with $f_{\F_v}^n(a) = x_0(v)$
    for some $n\geq 0$, we have $f_{\F_v}^{n_v}(a) = x_0(v)$.
    Since $f_{\F_v}^n(x_n(v))=x_0(v)$ for all $n \geq 0$,
    we have for all $m \geq 0$
    \begin{align}
         x_m(v) = f_{\F_v}^{n_v}(x_{m+n_v}(v)) = x_0(v).
    \end{align}
    Since this holds for all $v \notin S$, we get $x_m = x_0$ for all $m \geq 0$.
    This contradicts to $x_m \notin f^{-(m-1)}(y)$.
\end{proof}

\section{Preimages Question when the subvariety is not invariant}\label{sec:Preimages Question when the subvariety is not invariant}

Let $f \colon X \longrightarrow X$ be a surjective morphism on a projective variety.
Consider an arbitrary subvariety $Y \subset X$, not necessarily $f$-invariant.
One might expect that the preimages $f^{-s}(Y)$ are typically geometrically 
complicated and admit few rational points.
Stemming from this intuition, the following question 
(PIQ without the constraint of $Y$ being invariant) naturally arises.

\begin{question}\label{que:PIQ_non_invariant}
    Let $K$ be a number field.
    Let $X$ be a projective variety over $K$ and 
    $f \colon X \longrightarrow X$ a surjective morphism over $K$.
    Let $Y \subset X$ be a closed subscheme.
    Then is there $s_0 \in \Z_{\geq 0}$ such that for all $s \geq s_0$, 
    we have
    \begin{align}
        \biggl(f^{-s-1}(Y) \smallsetminus \bigcup_{0 \leq i \leq s} f^{-i}(Y) \biggr)(K) = \emptyset\ ?
    \end{align}
\end{question}

There are situations where the answer is negative by trivial reason, e.g., \ 
$f$ is an automorphism and $Y$ is a single rational point with infinite $f$-orbit.
It was somewhat surprising that
the question has a negative answer even if $f$ is polarized.
The following is a recipe for constructing counterexamples.

\begin{proposition}\label{prop:gpiq_ce_general_recipe}
    Let $K$ be a number field and
    $X$ a projective variety over $K$.
    Let $f,g \colon X \longrightarrow X$ be surjective morphisms over $K$
    and $Y \subset X$ be a closed subvariety.
    Suppose 
    \begin{enumerate}
        \item $f \circ g = g \circ f$;
        \item $f(g(Y))=Y$;
        \item $Y \cap g^n(Y) \subsetneq Y$ for all $n \geq 1$;
        \item $Y(K)$ is Zariski dense in $Y$.
    \end{enumerate}
    Then the answer to \cref{que:PIQ_non_invariant} for $X, f, Y$ 
    is negative.
\end{proposition}

\begin{proof}
    For all integers $n \geq m \geq 0$, we have
    \begin{align}
        f^m(g^n(Y)) = g^{n-m} \left((f\circ g)^m(Y)\right) = g^{n-m}(Y).
    \end{align}
    Thus for all $s \geq 0$, we have
    \begin{align}\label{ineq:rat_point_preimages_ce}
        &\biggl(f^{-s-1}(Y) \smallsetminus \bigcup_{0 \leq i \leq s} f^{-i}(Y) \biggr)(K)\\
        &\supset
        \biggl(g^{s+1}(Y) \smallsetminus \bigcup_{0 \leq i \leq s} f^{-i}(Y \cap g^{s+1-i}(Y)) \biggr) (K).
    \end{align}
    By assumption, $Y \cap g^{s+1-i}(Y)$ is a proper closed subset of $Y$.
    Since $Y$ is irreducible and $f^i, g^{s+1}$ are finite, $f^{-i}(Y \cap g^{s+1-i}(Y))$ does not contain $g^{s+1}(Y)$.
    Since $g^{s+1}(Y)(K)$ is Zariski dense in $g^{s+1}(Y)$,
    the right hand side of \cref{ineq:rat_point_preimages_ce} is non-empty.
\end{proof}

\begin{example}\label{ex:ce_to_gpiq_elliptic}
    There is a polarized endomorphism for which \cref{que:PIQ_non_invariant}
    is not true.
    Let $E$ be an elliptic curve over $\Q(i)$ such that
    $E_{\C} \simeq \C/(\Z + i\Z)$ 
    (cf.\ \cite[II, Example 1.3.1]{Silbook94} or \cite[I, Proposition 4.5]{Silbook94}).
    The endomorphisms defined by multiplication by $1+2i$ and $1-2i$ on $E_{\C}$ is defined over $\Q(i)$ (cf.\ \cite[II, Theorem 2.2 (b)]{Silbook94}). 
    Let $\f$, $\psi$ denote those endomorphisms on $E$ respectively.
    Then we have
    \begin{align}
        &\f \circ \psi = \psi \circ \f = [5]\\
        & \deg \f = \deg \psi = 5
    \end{align}
    where $[5]$ stands for the multiplication by $5$.
    We are going to apply \cref{prop:gpiq_ce_general_recipe}.
    Set $X = E \times E$, $f = \f \times \psi$, $g = \psi \times \f$,
    and $Y = \Delta \subset X$, the diagonal subvariety.
    Then, it is easy to see that
    \begin{align}
        &f \circ g = g \circ f\\
        &f(g(Y)) = Y.
    \end{align}
    There is a finite extension $K$ of $\Q(i)$ such that $Y(K) = E(K)$
    is infinite. We then replace everything with the base change to $K$.
    So far, we confirmed conditions (1)(2)(4).
    To see (3), it is enough to show that $\psi^n \neq \f^n$
    for all $n \geq 1$. 
    This follows from the fact that
    \begin{align}
        \frac{1+2i}{1-2i}
    \end{align}
    is not a root of $1$ as the unit group of $\Q(i)$
    is generated by $i$.
    We have proven that the morphism 
    $f \colon E \times E \longrightarrow E \times E$
    and the diagonal $\Delta \subset E \times E$ give a counterexample to
    \cref{que:PIQ_non_invariant}.
    Note that $f$ is polarized and \'etale.
\end{example}


\begin{example}
    In \cref{ex:ce_to_gpiq_elliptic}, $\f$ and $\psi$ commute with
    $[-1] \colon E \longrightarrow E$.
    Hence they induce endomorphisms on $\P^1_K = E/ \langle [-1] \rangle$.
    Let $\widetilde{\f}, \widetilde{\psi} \colon \P^1_K \longrightarrow \P^1_K$ be the induced endomorphisms.
    They commute and $\deg \widetilde{\f} = \deg \widetilde{\psi} = 5$.
    Set $X = \P^1_K \times \P^1_K$, $f = \widetilde{\f} \times \widetilde{\psi}$, $g =  \widetilde{\psi} \times \widetilde{\f}$,
    and $Y = \Delta \subset X$, the diagonal subvariety.
    Then we have
    \begin{align}
        f \circ g = g \circ f,\ 
        f(g(Y)) = Y,\ 
        \text{$Y(K)$ is Zariski dense in $Y$}.
    \end{align}
    Moreover, 
    \begin{align}
        g^n(Y) \subset Y \iff  \f^n = \pm\psi^n
        \iff \left(\frac{1+2i}{1-2i}\right)^n = \pm1
    \end{align}
    and hence $g^n(Y) \not\subset Y$ for $n \geq 1$.
    Therefore, by \cref{prop:gpiq_ce_general_recipe},
    $f \colon \P^1_K \times \P^1_K \longrightarrow \P^1_K \times \P^1_K$
    and $\Delta \subset \P^1_K \times \P^1_K$
    give a counterexample to \cref{que:PIQ_non_invariant}.
    Note that $f$ is polarized.
\end{example}

The above construction of counterexamples uses two mutually commutative 
self-morphisms.
Such a situation is somewhat special, and it may be reasonable to expect 
\cref{que:PIQ_non_invariant} holds for ``general" self-morphisms.
The following is one attempt in this direction.

\begin{theorem}\label{thm:GPIQ_P1P1}
    Let $K$ be a number field.
    Let $f,g \colon \P^1_K \longrightarrow \P^1_K$ be polynomial maps of the same degree $d \geq 2$.
    Then the answer to \cref{que:PIQ_non_invariant} for $f \times g \colon \P^1_K \times \P^1_K \longrightarrow \P^1_K \times \P^1_K$ and the diagonal subvariety
    $\Delta \subset \P^1_K \times \P^1_K$ is affirmative, i.e.\ 
    there is $s_0\in \Z_{\geq 0}$ such that for all $s \geq s_0$, we have
    \begin{align}
        \biggl((f\times g)^{-s-1}(\Delta) \smallsetminus \bigcup_{0 \leq i \leq s} (f \times g)^{-i}(\Delta) \biggr)(K) = \emptyset.
    \end{align}
\end{theorem}

The proof of this theorem is similar to the proof of
\cite[Theorem 1.6]{BMS23}.
We use the following straightforward generalization of
\cite[Lemma 4.3]{BMS23}.

\begin{lemma}\label{lem:etale_over_infty}
    Let $k$ be an algebraically closed field of characteristic zero.
    Let $P,Q \in k[x]$ be non-constant polynomials 
    such that $\deg Q \geq 2$ and $\deg Q \mid \deg P$.
    We use the same symbols $P, Q$ to denote the self-morphisms on $\P^1_k$ defined by $P, Q$.
    Consider the following diagram
    \begin{equation}
        \begin{tikzcd}
            \widetilde{X} \arrow[d,"\nu"] \arrow[ddd, "\pi_1", bend right=4em, swap] & \\
            X \arrow[r] \arrow[d,phantom, "\subset", sloped] & \Delta \arrow[d,phantom, "\subset", sloped] \\[-1em]
            \P^1_k \times \P^1_k \arrow[r,"P\times Q"] \arrow[d,"\pr_1"] & \P^1_k \times \P^1_k \arrow[d,"\pr_1"]\\
            \P^1_k \arrow[r,"P"] & \P^1_k
        \end{tikzcd}
    \end{equation}
    where  $\Delta$ is the diagonal subvariety,
    $X$ is an arbitrary irreducible component of $(P \times Q)^{-1}(\Delta)$ with reduced structure,
    $\nu$ is the normalization of $X$, and $\pi_1$ is the composite of $\nu$ and the first projection $\pr_1$.
    Then $\pi_1$ is \'etale over $\infty$.
\end{lemma}

\begin{proof}
    The proof is a straightforward modification of the proof of \cite[Lemma 4.3]{BMS23}.
    Let $\pi_2 = \pr_2 \circ \nu$ and $y = 1/x \in \O_{\P^1_k, \infty} \subset k(\P^1_k)$.
    Suppose there is a closed point $\a \in \widetilde{X}$ such that $\pi_1(\a) = \infty$
    and $\pi_1$ is ramified at $\a$.
    Let $e = e_{\pi_1,\a} (\geq 2)$ be the ramification index at $\a$.
    (We use this notation $e_{{\rm map},{\rm point}}$ for the ramification index in the following.)
    Since $P \circ \pi_1 = Q \circ \pi_2$, we have
    $Q(\pi_2(\a)) = P(\pi_1(\a)) = P(\infty) =\infty$.
    Thus $\pi_2(\a) = \infty$.
    Moreover, we have $e_{P,\infty}e_{\pi_1, \a} = e_{Q,\infty}e_{\pi_2,\a}$ and hence 
    \begin{align}
        e_{\pi_2, \a} = \frac{\deg P}{\deg Q} e = de
    \end{align}
    where $d := \deg P / \deg Q$.

    Let 
    \begin{align}
        \f := \pi_1^* y,\ \psi := \pi_2^* y. 
    \end{align}
    Note that we have $\f, \psi \in \O_{\widetilde{X},\a}$ and
    $v_{\a}(\f) = e$, $v_{\a}(\psi) = de$. Here $v_{\a}$ is the discrete valuation of
    $\O_{\widetilde{X},\a}$.
    Note also that $\f$ and $\psi$ generate $k(\widetilde{X})$ over $k$ as a field.
    By the relation $P\circ \pi_1 = Q \circ \pi_2$, we have
    \begin{align}
        &\pi_1^*(P^*y) = \pi_2^*(Q^*y),\ \text{i.e.}\\
        &\pi_1^*(1/P(1/y)) = \pi_2^*(1/Q(1/y)),\ \text{i.e.}\\
        &P(1/\f) = Q(1/\psi). \label{eq:Pf_equal_Qpsi}
    \end{align}

    Now, by Hensel, there is a uniformizer $u \in \widehat{\O}_{\widetilde{X},\a}$
    such that $\f = u^e$.
    We identify $\widehat{\O}_{\widetilde{X},\a} = k\lb u \rb$ and write
    \begin{align}
        \frac{1}{\psi} 
        = \sum_{\substack{n \geq -de \\ e\mid n}} a_n u^n 
        + \sum_{\substack{n \geq -de \\ e \not\mid n}} a_n u^n 
        \quad (a_n \in k,\ a_{-de} \neq 0).
    \end{align}
    We set
    \begin{align}
        f = \sum_{\substack{n \geq -de \\ e\mid n}} a_n u^n, \quad 
        g = \sum_{\substack{n \geq -de \\ e \not\mid n}} a_n u^n.
    \end{align}
    We claim $g = 0$.
    Indeed, assume $g \neq 0$ and let $m_0$ be the degree of the lowest degree term of $g$.
    By the definition of $g$, we have $m_0 \geq -de + 1$ and $e \not\mid m_0$.
    By \cref{eq:Pf_equal_Qpsi}, we have
    \begin{align}
        P(1/u^e) = Q(f + g) 
        = Q(f) + \sum_{i = 1}^{\deg Q} \frac{Q^{(i)}(f)}{i!} g^i.
    \end{align}
    The lowest degree of $\frac{Q^{(i)}(f)}{i!} g^i$ is $(-de)(\deg Q - i) + im_0$.
    Since $m_0 \geq -de +1$, we have
    \begin{align}
        (-de)(\deg Q - i) + im_0 > (-de)(\deg Q - 1) + m_0
    \end{align}
    for all $i \geq 2$.
    Thus, the lowest degree of 
    \begin{align}
        \sum_{i = 1}^{\deg Q} \frac{Q^{(i)}(f)}{i!} g^i
    \end{align}
    is $ (-de)(\deg Q - 1) + m_0$, which is not divisible by $e$.
    But this is a contradiction because 
    \begin{align}
        \sum_{i = 1}^{\deg Q} \frac{Q^{(i)}(f)}{i!} g^i = P(1/u^e) - Q(f)
    \end{align}
    is a Laurent series of $u^e$. This proves $g=0$ and, in particular,
    $\psi$ is a power series of $u^e$.

    Finally, since $\f$ and $\psi$ generate $k(\widetilde{X})$, there is a rational function 
    of $F(\f, \psi)$ of $\f$ and $\psi$ with coefficient in $k$ such that $F(\f, \psi)$ is a 
    uniformizer of $\O_{\widetilde{X},\a}$.
    But $F(\f, \psi)$ is a Lauren series of $u^e$ in $\widehat{\O}_{\widetilde{X},\a} = k\lb u \rb$,
    and hence, its valuation is divisible by $e$.
    This is a contradiction, and we complete the proof.
\end{proof}

\begin{proof}[Proof of \cref{thm:GPIQ_P1P1}]
    Let $\f = f \times g$.
    By \cite[Proposition 2.1]{Fak03}, there is $\d \in \Z_{\geq 1}$ such that the closed immersion 
    $i\colon \P^1_K \times \P^1_K \hookrightarrow \P^N_K$ defined by
    $\O_{\P^1_K \times \P^1_K}(\d,\d)$ is equivariant, i.e.\ 
    there is $F \colon \P^N_K \longrightarrow \P^N_K$ such that
    $F \circ i = i \circ \f$.
    The height $h$ of subvarieties of $\P^N_K$
    and the canonical height $\hh_F$ are defined as in 
    \cite[section 4]{Hut18}.
    The restriction of $h$ and $\hh_F$ to subvarieties of $\P^1_K \times \P^1_K$
    are denoted by $h$ (abuse of notation) and $\hh_\f$.
    We refer \cite[section 4]{Hut18} or \cite[Lemma 5.1]{BMS23} 
    for the properties of $h$ and $\hh_\f$ that we need for this proof.

    Now suppose \cref{que:PIQ_non_invariant} does not hold for
    $\f$ and $\Delta$.
    Then for all $m \in \Z_{\geq 0}$, the set
    \begin{align}
        \Sigma_{m+1} := \biggl(\f^{-m-1}(\Delta) \smallsetminus \bigcup_{0 \leq i \leq m} \f^{-i}(\Delta) \biggr)(K)
    \end{align}
    are infinite.
    Indeed, if there is $m \in \Z_{\geq 0}$ such that $\Sigma_{m+1}$ is finite,
    then the set
    \begin{align}
        \Sigma := \{ x \in (\P^1 \times \P^1)(K) \mid \text{$\f^l(x) \in \Sigma_{m+1}$ for some $l \in \Z_{\geq 0}$}  \}
    \end{align}
    is finite because it is a set of bounded heights.
    But we have
    \begin{align}
        \bigcup_{n \geq m+1} \Sigma_n \subset \Sigma
    \end{align}
    and the left-hand side is an infinite set by our hypothesis. This is a contradiction.

    For all $m \in \Z_{\geq 0}$, let
    \begin{align}
        S_{m+1} := 
        \left\{ C \ \middle|\ 
        \parbox{20em}{
        $C$ is an irreducible component of $\f^{-m-1}(\Delta)$
        which is not contained in  $\bigcup_{0 \leq i \leq m}\f^{-i}(\Delta)$
        and $C(K)$ is infinite
        }
        \right\}.
    \end{align}
    We set $S_0 = \{\Delta\}$.
    By the above argument, $S_m \neq \emptyset$ for all $m \geq 0$.
    Also $\f$ maps members of $S_m$ to members of $S_{m-1}$.
    Thus, they form an inverse system of non-empty finite sets.
    Let us take an element $(C_m)_{m \geq 0} \in \varprojlim S_m$.
    By the construction, $C_0 = \Delta$, 
    $C_{m+1} \not \subset \bigcup_{0 \leq i \leq m}\f^{-i}(\Delta)$,
    $\#C_m(K) = \infty$ for all $m \geq 0$,
    and we have the following commutative diagram:
    \begin{equation}
        \begin{tikzcd}
            \cdots \arrow[r]& C_{m+1} \arrow[r] \arrow[d,phantom,"\subset",sloped] & C_m \arrow[r] \arrow[d,phantom,"\subset",sloped] & \quad\cdots\quad \arrow[r] & \Delta \arrow[d,phantom,"\subset",sloped] \\[-1em]
            \cdots \arrow[r,"\f"]& \P^1_K \times \P^1_K \arrow[r,"\f"] & \P^1_K \times \P^1_K \arrow[r,"\f"] & \quad\cdots\quad \arrow[r,"\f"] & \P^1_K \times \P^1_K
        \end{tikzcd}
    \end{equation}
    We note that $C_m$'s are geometrically integral because $C_m(K)$ is dense in $C_m$.

    Let $\deg C_m := \deg \O_{\P^1 \times \P^1}(1,1)|_{C_m}$, i.e.\ the degree of $C_m$ with respect to
    the ample line bundle $\O_{\P^1 \times \P^1}(1,1)$.
    We claim that $\lim_{m \to \infty}\deg C_m = \infty$.
    Indeed, by \cite[Theorem 4.10]{Hut18} (cf.\ \cite[Lemma 5.1]{BMS23}),
    we have 
    \begin{align}
        \frac{\hh_\f(\Delta)}{\deg \O(\d,\d)|_{\Delta}} 
        = \frac{\hh_\f(\f^m(C_m))}{\deg \O(\d,\d)|_{\f^m(C_m)}}
        = d^m \frac{\hh_\f(C_m)}{\deg \O(\d,\d)|_{C_m}}.
    \end{align}
    Here $\f^m(C_m)$ is equipped with the reduced structure.
    This implies 
    \begin{align}
        \hh_\f(C_m) = \frac{\hh_\f(\Delta)}{\deg \O(\d,\d)|_{\Delta}} \frac{\d\deg C_m}{d^m}.
    \end{align}
    In particular, if there is a sequence $0 \leq m_1 < m_2 < \cdots$
    such that $\deg C_{m_i}$ is bounded, then $\hh_\f(C_{m_i})$ is also bounded.
    This implies the set of subvarieties $\{C_{m_i} \mid i \geq 1 \}$ is a finite set
    (cf.\ \cite{Hut18} or \cite[Lemma 5.1]{BMS23}).
    But this contradicts our construction of $C_m$.

    To end the proof, we prove that the genus of the normalization of
    $C_m \times_K \overline{K}$ is at least two if $m$ is large enough.
    We get a contradiction with the fact that $\# C_m(K)= \infty$ and the Faltings' theorem.

    We replace everything with the base change to $\overline{K}$ and 
    we work over $\overline{K}$.
    Let us decompose 
    \begin{align}
        f = \a_1 \circ \b_1,\ g = \a_2 \circ \b_2
    \end{align}
    where $\a_1$ and $\a_2$ are indecomposable polynomial maps of degree at least two and 
    $\b_1$ and $\b_2$ are polynomial maps.
    Here a surjective morphism $\P^1 \longrightarrow \P^1$ is said to be indecomposable if
    the induced function field extension admits no non-trivial intermediate fields.
    Consider the following commutative diagram.
    \begin{equation}
        \begin{tikzcd}
            C_{m+1} \arrow[r] \arrow[d,phantom,"\subset",sloped]  & X \arrow[r] \arrow[dd] & C_m \arrow[d,phantom,"\subset",sloped] \arrow[r] \arrow[dd,bend right=5em, "\nu_i",swap] & \quad \cdots \quad \arrow[r]& \Delta \arrow[d,phantom,"\subset",sloped] \\[-1em]
            \P^1 \times \P^1 \arrow[d,"\pr_i"]& & \P^1 \times \P^1 \arrow[d,"\pr_i"] \arrow[r, "\f"] & \quad \cdots \quad \arrow[r, "\f"]& \P^1 \times \P^1\\
            \P^1 \arrow[r, "\b_i"] & \P^1 \arrow[r, "\a_i"] & \P^1 &
        \end{tikzcd}
    \end{equation}
    Here $i$ is $1$ or $2$, $\nu_i$ is the composition of the inclusion and the projection, and
    $X$ is the reduced scheme of the fiber product of 
    $C_m \xlongrightarrow{\nu_i} \P^1$ and $\P^1 \xlongrightarrow{\a_i} \P^1$.
    The fiber product's universal property induces the morphism $C_{m+1} \longrightarrow X$ is induced by the fiber product's universal property.
    Let $\mu \colon \widetilde{C}_m \longrightarrow C_m$ be the normalization.
    By \cref{lem:etale_over_infty}, $\nu_i \circ \mu$ is \'etale over $\infty$.
    Thus by \cite[Proposition 4.1]{BMS23}, $X$ is irreducible and
    \begin{align}\label{ineq:genus_lower_bound}
        2 g_X - 2 \geq \deg \a_i (2 g_{C_m} - 2) + (\deg \a_i - 1) \deg \nu_i
    \end{align}
    where $g_X$ and $g_{C_m}$ are the genus of the normalization of $X$ and $C_m$ respectively.
    Now note that $\deg C_m = \deg \nu_1 + \deg \nu_2$.
    Since $\deg C_m \to \infty$, we may take $m$ large enough so that
    $\deg \nu_i \geq 5$ for $i=1$ or $2$.
    Then by \cref{ineq:genus_lower_bound},
    \begin{align}
        2 g_X - 2 \geq -2 \deg \a_i + 5(\deg \a_i -1) = 3 \deg \a_i - 5 \geq 1. 
    \end{align}
    Therefore $g_{C_{m+1}} \geq g_X \geq 2$ and we are done.
\end{proof}

\section{\'Etale morphisms}\label{sec:Etale morphisms}

\begin{theorem}\label{thm:PIQ-etale}
Let $K$ be a number field.
Let $X$ be a projective variety over $K$ and 
$f \colon X \longrightarrow X$ be an \'etale morphism over $K$.
Let $Y \subset X$ be a closed subscheme such that $f|_{Y}$ factors through $Y \subset X$.
Let $d \in \Z_{\geq 1}$.
Then there exists a non-negative integer $s_{0}$ such that 
\[
(f^{-s-1}(Y)\smallsetminus f^{-s}(Y))(L)=\emptyset
\]
for all field extension $K \subset L$ with $[L:K] \leq d$ and $s\geq s_{0}$.
\end{theorem}

\begin{remark}
    Let us remark that \cref{thm:PIQ-etale} does not automatically follow from
    (the proof of) \cref{prop:twoPIQequiv} and \cite[Theorem 2.1]{BMS23}, which is $d=1$ version of
    \cref{thm:PIQ-etale}.
    Indeed, in the proof of \cref{prop:twoPIQequiv}, we apply PIQ (with fixed number field) to
    the self-morphism induced on the symmetric product $X^{(d)}$.
    But this induced morphism could be non-\'etale. 
\end{remark}

The following lemma is the key in the proof of \cref{thm:PIQ-etale}.

\begin{lemma}\label{lem:inverse-images-by-etale}
Let $f \colon X \longrightarrow X$ be an \'etale morphism 
of a projective variety $X$ over a number field $K$.
Let $Y \subset X$ be a normal closed subscheme.
Let $d\in \Z_{\geq 1}$.
Then there is a non-negative integer $t$ such that for all $s \geq 0$, we have
\begin{align}
    \# \left\{ Z\  \middle|\
    \parbox{20em}{$Z$ is an irreducible component of $f^{-s}(Y)$ such that $Z(L) \neq  \emptyset$ for some $L$ with $[L:K] \leq d$.}  \right\}
    \leq t
\end{align}

\end{lemma}
\begin{proof}
As in \cite[Lemma 2.2]{BMS23},
let $S$ be a sufficiently large finite set of primes of $K$ so that we have models $\X, F$, and $\Y$
over $\O_{K,S}$ of $X, f$, and $Y$ respectively with the following properties:
\begin{enumerate}
\item $\X$ is an integral projective scheme over $\O_{K,S}$;
\item $F \colon \X \longrightarrow \X$ is an \'etale surjective morphism over $\O_{K,S}$;
\item $\Y$ is a closed subscheme of $\X$ and normal.
\end{enumerate}

Since $F$ is \'etale and $\Y$ is normal, $F^{-s}(\Y)$ is a disjoint union of normal integral schemes for 
any $s \geq 0$:
\begin{align}
F^{-s}(\Y) = \coprod_{i=1}^{r_s} \mathcal{Z}_{i}^{(s)}.
\end{align}
Let $Z_{i}^{(s)}$ be the generic fiber of $ \mathcal{Z}_{i}^{(s)}$. 
Then $Z_{i}^{(s)}$'s are the irreducible components of $f^{-s}(Y)$:
\begin{align}
f^{-s}(Y) = \coprod_{i=1}^{r_s} Z_{i}^{(s)}.
\end{align}
By relabeling if necessary, we may assume that the first $l$ components admit rational points of 
degree at most $d$ over $K$. That is, for $i = 1,\dots ,l$, there are finite extensions $K \subset L_i$
with $[L_i : K] \leq d$ and
\begin{align}
    Z_i^{(s)}(L_i) \neq \emptyset.
\end{align}
Let $R_i$ be the integral closure of $\O_{K,S}$ in $L_i$.
Then any point of $Z_i^{(s)}(L_i)$ extends to a point of $\mathcal{Z}_i^{(s)}(R_i)$
because $\mathcal{Z}_i^{(s)}$ is proper over $\O_{K,S}$.
Let us take $x_i \in \mathcal{Z}_i^{(s)}(R_i)$ for each $i = 1,\dots, l$.
Pick a maximal ideal $\p \subset \O_{K,S}$ and take maximal ideals 
$\q_i \subset R_i$ containing $\p R_i$.
Then we get the following diagram:
\begin{equation}
    \begin{tikzcd}
        \Spec R_i \arrow[d,"x_i"] & \Spec R_i/\p R_i \arrow[l] \arrow[d] & \Spec R_i/\q_i \arrow[l] \arrow[ld,"\xi_i",pos=0.4] \\
        \X \arrow[d] & \X \times_{\O_{K,S}} \O_{K,S}/\p \mathrlap{=:\X_{\p}} \arrow[l] \arrow[d]\\
        \Spec \O_{K,S} & \Spec \O_{K,S}/\p \arrow[l]
    \end{tikzcd}
\end{equation}
where the squares are fiber products and $\xi_i$ is the induced morphism.
We have
\begin{align}
    [R_i/\q_i : \O_{K,S}/\p] \leq [L_i : K] \leq d.
\end{align}
In particular, if $\# \O_{K,S}/\p = q$, then $\# R_i/\q_i = q^j$ for some $1 \leq j \leq d$.
Also the images of $\xi_i$'s are disjoint because $\mathcal{Z}_i^{(s)}$ are disjoint.
Therefore, we get
\begin{align}
    l \leq \# \bigcup_{j=1}^d \Hom(\Spec \F_{q^j}, \X_{\p})
\end{align}
where the $\Hom$ in the right-hand side stands for the set of morphisms of schemes.
We are done since the right-hand side is finite and independent of $s$.
\end{proof}

\begin{proof}[Proof of \cref{thm:PIQ-etale}]
The proof is the same as that of \cite[Theorem 2.1]{BMS23}.
First, we may assume $Y$ is reduced.
If $Y$ is normal, the statement follows from \cref{lem:inverse-images-by-etale}.

For the general case, we proceed by induction on $\dim Y$.
If $\dim Y=0$, then $Y$ is normal, and we are done.

Suppose $\dim Y > 0$.
Let us write $Y = Y_{1} \cup Y_{2}$ where
$Y_{1}$ is the union of all maximal dimensional irreducible components and
$Y_{2}$ is the union of other components.
We endow $Y_{1}, Y_{2}$ with the reduced scheme structure. 
Then as in the proof of \cite[Theorem 2.1]{BMS23}, we have 
\begin{align}
f(Y_{1}) \subset Y_{1}, \quad 
f(\Sing Y_{1}) \subset \Sing Y_{1}
\end{align}
where $\Sing Y_{1}$ is the non-regular locus of $Y_{1}$.

Let $\nu \colon \widetilde{Y_{1}} \longrightarrow Y_{1}$ be the normalization,
i.e.\ $ \widetilde{Y_{1}}$ is the union of the normalization of irreducible components of $Y_{1}$.
Take a closed immersion $i \colon \widetilde{Y_{1}} \longrightarrow \P^{N}_{K}$ and
we regard $ \widetilde{Y_{1}}$ as a closed subscheme of $\P^{N} \times X$ via
\[
\widetilde{Y_{1}} \xrightarrow{(i, \nu)} \P^{N} \times Y_{1} \subset \P^{N} \times X.
\]
Set $g=\id \times f \colon \P^{N}\times X \longrightarrow \P^{N}\times X$.
Since $g$ is \'etale and $ \widetilde{Y_{1}}$ is normal, by \cref{lem:inverse-images-by-etale},
there exists a positive integer $t$ such for all $s$, at most $t$ irreducible components of $g^{-s}(\widetilde{Y_{1}})$
admit $L$-points for some $L$ with $[L:K] \leq d$. 

Note that the diagram 
\begin{equation}
    \begin{tikzcd}
        g^{-s-1}(\widetilde{Y_{1}}) \arrow[r,"g^{s+1}"] \arrow[d]& \widetilde{Y_{1}} \arrow[d,"\nu"]\\
        f^{-s-1}(Y_{1}) \arrow[r,"f^{s+1}",swap] & Y_{1}
    \end{tikzcd}
\end{equation}
is cartesian and $\nu$ is an isomorphism over $Y_{1} \smallsetminus \Sing Y_{1}$.

We apply the induction hypothesis to $f$ and $\Sing Y_1$.
Then there exists a non-negative integer $s_{0}$ such that 
\begin{align}\label{indhyp-to-sing}
    (f^{-s-1}(\Sing Y_{1}) \smallsetminus f^{-s}(\Sing Y_{1}))(L) = \emptyset
\end{align}
for all $L$ with $[L:K] \leq d$ and $s \geq s_0$.
We claim that 
\begin{align}
    (f^{-s_0-t-1}(Y_{1}) \smallsetminus f^{-s_0-t}(Y_{1}))(L) = \emptyset
\end{align}
for all $L$ with $[L:K] \leq d$.
Indeed, suppose there is $L$ with $[L:K] \leq d$ and
$x \in (f^{-s_0-t-1}(Y_{1}) \smallsetminus f^{-s_0-t}(Y_{1}))(L)$.
Then $x, f(x), \dots, f^t(x)$ are contained in different irreducible components of
$f^{-s_0-t-1}(Y_1)$.
Moreover, they are not mapped into $\Sing Y_1$ by $f^{s_0+t+1}$ because 
$x \not\in f^{-s_0-t}(Y_1)$ and \cref{indhyp-to-sing}.
Thus $x, f(x), \dots, f^t(x)$ are contained in the open subset over which 
$g^{-s_0-t-1}(\widetilde{Y}_1) \longrightarrow f^{-s_0-t-1}(Y_1)$ is isomorphic.
Thus $x, f(x), \dots , f^t(x)$ produce $L$-points of $g^{-s_0-t-1}(\widetilde{Y}_1)$
that are on distinct irreducible components. This contradicts the definition of $t$.

Now we have proven that
\[
(f^{-s-1}(Y_{1}) \smallsetminus f^{-s}(Y_{1}))(L)= \emptyset
\]
for all $L$ with $[L:K] \leq d$ and $s \geq s_0 + t$.

Let $Y_{3}\subset Y_{2}$ be the union of all irreducible components, which are eventually mapped into $Y_{1}$ and
let $Y_{4} \subset Y_{2}$ be the union of all other components.
Then we have $f(Y_{4}) \subset Y_{4}$.
Take $r>0$ so that $f^{r}(Y_{3})\subset Y_{1}$.
Then for $s \geq s_{0}+t$ and for all $L$ with $[L:K] \leq d$,
\begin{align}
&\hphantom{=} (f^{-s-1}(Y)\smallsetminus f^{-s}(Y))(L) \\
&\subset (f^{-s-1}(Y_{1})\smallsetminus f^{-s}(Y_{1}))(L) \cup (f^{-s-1}(Y_{3})\smallsetminus f^{-s}(Y))(L) \\
& \qquad \qquad \qquad \qquad \qquad \qquad \qquad \qquad \cup (f^{-s-1}(Y_{4})\smallsetminus f^{-s}(Y_{4}))(L)\\
& = (f^{-s-1}(Y_{3})\smallsetminus f^{-s}(Y))(L) \cup (f^{-s-1}(Y_{4})\smallsetminus f^{-s}(Y_{4}))(L)\\
& \subset (f^{-s-1 -r}(Y_{1})\smallsetminus f^{-s}(Y))(L) 
\cup (f^{-s-1}(Y_{4})\smallsetminus f^{-s}(Y_{4}))(L)\\
&=(f^{-s-1}(Y_{4})\smallsetminus f^{-s}(Y_{4}))(L).
\end{align}
Applying the induction hypothesis to $Y_{4}$ and $f$, we are done.
\end{proof}

\section{Uniformity of \texorpdfstring{$s_0$}{s0} for abelian varieties}\label{sec:uniformity of s0 ab var}

In this section, we prove \cref{thm:main_uniformity_ab_var}.
At its core, the uniformity arises from the finiteness of torsion points 
defined over a number field, and 
a finiteness results on abelian subvarieties due to Lenstra and Oort \cite{Len96}.
Let us start with a linear algebra lemma.

\begin{lemma}\label{lem:period-of-subspaces}
    Let $V$ be a finite-dimensional $\Q$-vector space.
    Then there exists $n_0 \in \Z_{\geq 1}$ that depends only on $\dim V$
    with the following property.
    For any $\Q$-linear automorphism $f \colon V \longrightarrow V$ and 
    any $f$-periodic $\Q$-subspace $W \subset V$, we have
    \begin{align}
        f^{n_0}(W) = W.
    \end{align}
    Here a subspace $W$ is said to be $f$-periodic if
    $f^n(W) = W$ for some $n \in \Z_{\geq 1}$.
\end{lemma}
\begin{proof}
    Let $W$ be an $f$-periodic $\Q$-subspace of $V$ and let
    $d = \dim W$.
    If $d = 0$, then $f(W) = W$.
    Suppose $d \geq 1$.
    Consider the $\Q$-linear automorphism 
    \begin{align}
        \wedgeop{d} f \colon \wedgeop{d} V \longrightarrow \wedgeop{d} V.
    \end{align}
    Then, it is easy to see that
    \begin{align}
        f^n(W) = W \iff 
        (\wedgeop{d} f)^n(\wedgeop{d} W) = \wedgeop{d} W
    \end{align}
    for all $n \in \Z_{\geq 1}$.
    Therefore, it is enough to show the statement for one-dimensional subspaces $W$.

    Now, suppose $W \subset V$ is an $f$-periodic one-dimensional subspace.
    We will show that there is $n_0 \in \Z_{\geq 1}$ depending only on $\dim V$
    such that $f^{n_0}(W) = W$.
    Take $n \in \Z_{\geq 1}$ such that $f^n(W) = W$.
    Let $w \in W \smallsetminus \{0\}$ and set
    \begin{align}
        U = \langle w, f(w), \dots, f^{n-1}(w) \rangle_{\Q},
    \end{align}
    the subspace of $V$ generated by $w, \dots, f^{n-1}(w)$.
    We have $f(U)= U$ and there is $a \in \Q^{\times}$ such that
    $f|_U^n$ is multiplication by $a$.
    In particular, $f|_U$ is diagonalizable over $\C$.
    Replacing $n$ with $2n$ if necessary, we may assume $a >0$.
    Take $b \in \R_{>0}$ such that $b^n = a$.
    Let $\l_1,\dots, \l_r$ be all the roots of the characteristic polynomial 
    of $f|_U$ (possibly repeated).
    Note $r = \dim U \leq \dim V$.
    Then we have $\l_i^n = a$, and hence we can write 
    \begin{align}
        \l_i = \zeta_i b
    \end{align}
    where $\zeta_i \in \C$ with $\zeta_i^n=1$.
    Then
    \begin{align}
        \Q \ni \l_1 \cdots \l_r = \zeta_1 \cdots \zeta_r b^r
    \end{align}
    implies $b^r \in \Q$.
    Thus 
    \begin{align}
        [\Q(\zeta_i) : \Q] = [\Q(b^{-1}\l_i) : \Q] \leq r^2 \leq (\dim V)^2.
    \end{align}
    Therefore, there is $n_0 \in \Z_{\geq 1}$ depending only on $\dim V$
    such that 
    \begin{align}
        \l_1^{n_0} = \cdots = \l_r^{n_0} \in \Q. 
    \end{align}
    Remembering $f|_U$ is diagonalizable, we can deduce that $f|_U^{n_0}$
    is a multiplication by a non-zero rational number, and hence 
    $f^{n_0}(W) = W$.
\end{proof}

We can uniformly bound the period of periodic abelian subvarieties.

\begin{lemma}\label{lem:period-of-absubvar}
    Let $X$ be an abelian variety over a number field $K$.
    Then, there exists $n_0 \in \Z_{\geq 1}$ that depends only on 
    $X$ with the following property.
    For any isogeny $\f \colon X \longrightarrow X$ 
    and any $\f$-periodic abelian subvariety $A \subset X$,
    we have
    \begin{align}
        \f^{n_0}(A) = A.
    \end{align}
    Here $A$ being $\f$-periodic means that 
    $\f^n(A) = A$ for some $n \in \Z_{\geq 1}$.
\end{lemma}
\begin{proof}
    By \cite{Len96}, the action of group automorphisms of $X$ defined over $K$
    on the set of abelian subvarieties of $X$ has only finitely many orbits.
    Let $A_1, \dots , A_n \subset X$ be abelian subvarieties that represent 
    these finitely many orbits.
    Then there is a finite extension $L$ of $K$ such that $A_i(L)$ is Zariski dense 
    in $(A_i)_L$ for all $i=1,\dots ,n$.
    Since every abelian subvarieties of $X$ are isomorphic to one of $A_i$
    over $K$, it follows that every abelian subvariety has densely many
    $L$-rational points.
    Replacing $K$ with $L$ and $X$ with $X_L$, it is enough to prove the statement 
    under the additional assumption that $A(K)$ is Zariski dense in $A$.
    
    Let $V = X(K) \otimes_{\Z} \Q$.
    Then, $V$ is a finite-dimensional $\Q$-vector space.
    Let $\f$ and $A$ be as in the statement.
    Then $\f$ induces an $\Q$-linear automorphism $\f_\Q \colon V \longrightarrow V$,
    and $A(K) \otimes_{\Z} \Q \subset V$ is a $\f_\Q$-periodic subspace.
    By \cref{lem:period-of-subspaces}, there is $n_0 \in \Z_{\geq 1}$
    depending only on $\dim V$ such that 
    \begin{align}
        \f_{\Q}^{n_0}(A(K) \otimes_{\Z} \Q) = A(K) \otimes_{\Z} \Q.
    \end{align}
    This implies there is $m \in \Z_{\geq 1}$ such that
    \begin{align}
        m \f^{n_0} (A(K)) \subset A(K),
    \end{align}
    and therefore $\f^{n_0}(mA(K)) \subset A(K)$.
    Since $A(K)$ is Zariski dense in $A$,
    $mA(K)$ is also Zariski dense in $[m](A) = A$.
    (Here $[m]$ stands for the multiplication by $m$ map acting on the 
    scheme $A$.)
    Taking Zariski closure, we get
    \begin{align}
        \f^{n_0}(A) = \overline{\f^{n_0}(mA(K))} \subset A.
    \end{align}
    Looking at the dimension, we get $\f^{n_0}(A) = A$ and we are done.
\end{proof}

We summarize some facts on minimal polynomials of endomorphisms on abelian varieties
that we use later.
These are well-known, but we give proofs for completeness.

\begin{lemma}\label{lem:min-poly}
    Let $k$ be a field and $X$ be an abelian variety over $k$.
    For a group endomorphism $\f \colon X \longrightarrow X$,
    the minimal polynomial of $\f$ is denoted by $P_\f$.
    ($P_\f$ is a monic polynomial with integer coefficients.)
    \begin{enumerate}
        \item For a group endomorphism $\f \colon X \longrightarrow X$ and 
        $n \in \Z_{\geq 1}$, 
        $P_\f$ has a cyclotomic polynomial as a factor if and only if 
        so  does $P_{\f^n}$.

        \item There are $n_0, m_0 \in \Z_{\geq 1}$ depending only on $\dim X$
        such that for every group endomorphism $\f \colon X \longrightarrow X$
        whose minimal polynomial is
        a product of cyclotomic polynomials, we have $(\f^{n_0} - \id)^{m_0}=0$.

        \item There is $n_0 \in \Z_{\geq 1}$ depending only on $\dim X$
        such that for every group endomorphism $\f \colon X \longrightarrow X$,
        the minimal polynomial of $\f^{n_0}$ is of the form
        \begin{align}
            P_{\f^{n_0}}(t) = (t-1)^m Q(t)
        \end{align}
        for some $m \in \Z_{\geq 0}$ and $Q(t) \in \Z[t]$ such that
        $Q(t)$ does not have cyclotomic polynomials as a factor.
    \end{enumerate}
\end{lemma}
\begin{proof}
    (1)
    Note that $P_\f$ has a cyclotomic factor if and only if
    there is $m \in \Z_{\geq 1}$ such that
    $\f^m -\id$ is not surjective. Indeed, let us write $P_\f = Q(t) R(t)$
    where $Q, R \in \Z[t]$ are monic and $Q$ has no cyclotomic factor and 
    $R$ is a product of cyclotomic polynomials (possibly repeated).
    Suppose $R \neq 1$.
    Take $l, m \in \Z_{\geq 1}$ such that $R(t) \mid (t^m -1)^l$.
    Since $R \neq 1$, $R(\f)$ is not surjective and hence 
    $(\f^m - \id)^l$ is not surjective.
    This implies $\f^m - \id$ is not surjective.
    Conversely, suppose $\f^m -\id$ is not surjective for some $m \in \Z_{\geq 1}$.
    If $R = 1$, then $P_\f(t)$ and $t^m - 1$ are coprime.
    This implies $\f^m - \id$ is surjective and contradiction.

    Now, the statement easily follows.

    (2)
    Since 
    \begin{align}
        \Q[t]/(P_\f(t)) \longrightarrow \End(X)_{\Q}
    \end{align}
    is an injection of $\Q$-vector spaces, we have
    \begin{align}
        \deg P_\f \leq \dim \End(X)_{\Q} \leq 4(\dim X)^2.
    \end{align}
    Thus there exists $n_0 \in \Z_{\geq 1}$ that is determined by 
    $\dim X$ such that all the roots of $P_\f$ are $n_0$-th root
    of $1$.
    Hence $P_\f(t) \mid (t^{n_0} - 1)^{4(\dim X)^2}$
    and we are done.

    (3)
    Let us write $P_\f = Q R$ as in the proof of $(1)$.
    Let $A = Q(\f)(X)$ and $B = R(\f)(X)$.
    Then $A, B \subset X$ are $\f$-invariant abelian subvarieties.
    Moreover, we have the commutative diagram
    \begin{equation}
        \begin{tikzcd}
            A \times B \arrow[r,"\f_A \times \f_B"] \arrow[d] 
            & A \times B \arrow[d]\\
            X \arrow[r,"\f",swap] & X
        \end{tikzcd}
    \end{equation}
    where the vertical arrows are additions and $\f_A$ and $\f_B$
    are restrictions of $\f$ on $A$ and $B$ respectively.
    By construction, the vertical arrow is an isogeny.
    (cf.\ \cite[section 3]{Sil17}.)

    The minimal polynomial of $\f_A$ is $R(t)$.
    By (2), there are $n_0, m_0$ depending only on $\dim X$ such that
    $(\f_A^{n_0} - 1)^{m_0}=0$.
    Thus $P_{\f_A^{n_0}}(t) \mid (t - 1)^{m_0}$ and hence
    $P_{\f_A^{n_0}}(t) = (t-1)^m$ for some $m \in \Z_{\geq 1}$.
    On the other hand, the minimal polynomial of $\f_B$ is $Q$,
    which does not have cyclotomic factors.
    By (1), $P_{\f_B^{n_0}}$ has no cyclotomic factor as well.
    Therefore $P_{\f_A^{n_0}}$ and $P_{\f_B^{n_0}}$ are coprime.
    Thus 
    \begin{align}
        P_{\f^{n_0}} = P_{(\f_A \times \f_B)^{n_0}} 
        = P_{\f_A^{n_0}} P_{\f_B^{n_0}} = (t-1)^mP_{\f_B^{n_0}}
    \end{align}
    and we are done.
\end{proof}

Now we are ready to prove \cref{thm:main_uniformity_ab_var}.

\begin{theorem}[\cref{thm:main_uniformity_ab_var}]\label{thm:uniform-PIQ-ab-var}
    Let $X$ be an abelian variety over a number field $K$.
    Then there exists $s_0 \in \Z_{\geq 0}$ depending only on $X$ with the 
    following property.
    For any surjective morphism $f \colon X \longrightarrow X$ and 
    any reduced closed subscheme $Y \subset X$ with $f(Y) \subset Y$, we have 
    \begin{align}\label{eq:uniform_PIQ_ab_var_diff_set}
        (f^{-s_0 -1}(Y) \smallsetminus f^{-s_0}(Y))(K) = \emptyset.
    \end{align}
\end{theorem}

\begin{proof}
    We prove the statement by induction on $\dim X$.
    When $\dim X = 0$, there is nothing to prove.
    Suppose $\dim X \geq 1$.

    Let $f$ and $Y$ be as in the statement.
    If $Y(K) = \emptyset$, then \cref{eq:uniform_PIQ_ab_var_diff_set} holds for all $s_0 \in \Z_{\geq 0}$.
    Suppose $Y(K) \neq \emptyset$.
    Let $Z$ be the Zariski closure of $Y(K)$.
    Then we have $f(Z) \subset Z$ and 
    \begin{align}
        \left( f^{-s -1}(Y) \smallsetminus f^{-s}(Y) \right)(K) 
        = \left( f^{-s -1}(Z) \smallsetminus f^{-s}(Z) \right)(K)
    \end{align}
    for all $s \in \Z_{\geq 0}$.
    Thus replacing $Y$ with $Z$, we may assume $Y(K)$
    is Zariski dense in $Y$.
    Note that in this case, every irreducible component of $Y$
    has Zariski dense set of $K$-rational points and in particular
    geometrically irreducible.
    Let $Y = Y_1 \cup \cdots \cup Y_r$ be the irreducible decomposition 
    so that $Y_1, \dots, Y_{r'}$ are $f$-periodic and
    $Y_{r'+1},\dots , Y_r$ are mapped into periodic ones by some iterates of $f$. Take $m \in \Z_{\geq 1}$ so that for all 
    $j \in \{r'+1,\dots, r \}$, $f^m(Y_j) \subset Y_i$ for some 
    $i \in \{1, \dots, r'\}$.
    Let $Y' = Y_1 \cup \cdots \cup Y_{r'}$.
    Then we have
    \begin{align}
        f^{-s-1}(Y)(K)
        \subset
        f^{-s-1-m}(Y')(K)
    \end{align}
    for all $s \in \Z_{\geq 0}$.
    Thus, we may replace $Y$ with $Y'$ to prove the statement.
    In particular, we may assume every irreducible component of $Y$
    is $f$-periodic.
    Moreover, since it is enough to prove the existence of $s_0$ for each
    $f$-cycles of irreducible components,
    we may further assume that irreducible components of $Y$
    form a cycle under $f$, say, $f(Y_i) = Y_{i+1}$ for $i=1,\dots , r-1$
    and $f(Y_r) = Y_1$.
    Since $Y_1(K) \neq \emptyset$, picking a point $y \in Y_1(K)$
    and taking conjugate by $T_{y} \colon X \longrightarrow X$,
    we may assume $0 \in Y_1$.
    Finally, let us write $f = T_a \circ \f$ with $a \in X(K)$
    and a surjective group endomorphism $\f$. 
    By \cref{lem:min-poly}(3), replacing $f$ with iterates depending only
    on $\dim X$, we may assume the cyclotomic factor of the
    minimal polynomial of $\f$ is of the form $(t-1)^d$ for some $d\in \Z_{\geq 0}$.

    In summary, it is enough to prove the existence of $s_0$ in the following setting:
    \begin{itemize}
        \item $f = T_a \circ \f$ where $P_\f(t) = (t-1)^dQ(t)$ for some 
        $d \in \Z_{\geq 0}$ and $Q$ having no cyclotomic factors;
        \item $Y(K) \subset Y$ is Zariski dense;
        \item We have an irreducible decomposition $Y = Y_1 \cup\cdots \cup Y_r$ such that
        $0 \in Y_1$ and $f(Y_i) = Y_{i+1}$ for $i=1, \dots, r-1$ and $f(Y_r)= Y_1$.
    \end{itemize}
    
    Let $A := \Stab(Y_1)^{\circ}$ be the identity component of the stabilizer of
    $Y_1$ in $X$. This is an abelian subvariety.
    Applying $f$ to the equality $Y_1 + A = Y_1$, we get $Y_2 + \f(A) = Y_2$.
    Repeating this, we get $Y_1 + \f^r(A) = Y_1$.
    In particular, we get $\f^r(A) = A$.
    By \cref{lem:period-of-absubvar}, there exists $n_0 \in \Z_{\geq 1}$ depending 
    only on $X$ such that $\f^{n_0}(A) = A$.

    We claim that we may replace $f$ with $f^{n_0}$ and 
    $Y$ with the $f^{n_0}$-orbit of $Y_1$.
    Indeed, let $W$ be the $f^{n_0}$-orbit of $Y_1$ and suppose that
    there is $s_0 \in \Z_{\geq 0}$ depending only on $X$ such that
    \begin{align}
        ((f^{n_0})^{-s_0 -1}(W) \smallsetminus (f^{n_0})^{-s_0}(W))(K) = \emptyset.
    \end{align}
    Take $x \in X(K)$ such that $f^s(x) \in Y$ for some $s\in \Z_{\geq 0}$.
    Then $f^s(x) \in W$ for some $s\in \Z_{\geq 0}$ as well.
    Write $s = q n_0 + l$ with $q \in \Z_{\geq 0}$ and $0 \leq l \leq n_0 -1$.
    Then $(f^{n_0})^q(f^l(x)) \in W$ implies $(f^{n_0})^{s_0}(f^l(x)) \in W$
    and therefore $f^{n_0s_0 + n_0 -1}(x) \in Y$.

    Now replacing $f$ with $f^{n_0}$ and $Y$ with $W$,
    we may assume $\f(A) = A$.
    Note that in this case, we have $\Stab(Y_i)^{\circ} = A$ for all $i=1,\dots, r$.
    Consider 
    \begin{equation}
        \begin{tikzcd}
            X \arrow[r,"f"] \arrow[d, "\pi", swap] &[1em] X \arrow[d,"\pi"] \\
            X / A 
            \arrow[r,"\overline{f} := T_{\pi(a)} \circ \overline{\f}" {yshift=-1ex},swap] & 
            X/A
        \end{tikzcd}
    \end{equation}
    where $\pi$ is the quotient morphism and
    $\overline{\f}$ is the group endomorphism induced by $\f$.
    Since $\pi(Y_1)(K)$ is Zariski dense in $\pi(Y_1)$, $\pi(Y_1)$
    is geometrically irreducible (cf.\ \cite[Proposition 2.3.26]{Poonen_book}).

    We also have $\Stab(\pi(Y_1)_{\overline{K}})^{\circ} = \Stab(\pi(Y_1))^{\circ}_{\overline{K}} = 0$.
    Thus by \cite[(3.7) Theorem]{Mo85}, $\pi(Y_1)_{\overline{K}}$
    is a variety of general type (possibly a single point).
    As $\pi(Y_1)$ is $\overline{f}$-periodic, there is $m \in \Z_{\geq 1}$
    such that $\overline{f}^m|_{\pi(Y_1)} = \id_{\pi(Y_1)}$.
    Applying this to $0 \in \pi(Y_1)$, we get
    \begin{align}\label{eq:sum-of-fpii}
        \sum_{i=0}^{m-1}\overline{\f}^i(\pi(a)) = 0.
    \end{align}
    In particular, we get $\overline{\f}^m|_{\pi(Y_1)} = \id_{\pi(Y_1)}$.
    Let $H \subset X/A$ be the abelian subvariety generated by
    the $\overline{\f}$-orbit of $\pi(Y_1)$.
    Then $H$ is $\overline{\f}$-invariant and
    $\overline{\f}|_H^m = \id_H$.
    Thus $P_{\overline{\f}|_H}(t) \mid t^m - 1$ and we also have
    $P_{\overline{\f}|_H} \mid P_{\overline{\f}} \mid P_{\f}$.
    Since the cyclotomic factor of $P_\f$ is $(t-1)^d$, we get 
    $P_{\overline{\f}|_H}(t) = t-1$, in other words, 
    $\overline{\f}|_H = \id_H$.

    Now let $B = (X/A) / (\overline{\f} -\id)(X/A)$.
    
    {\bf Case 1}. Suppose $\dim B \geq 1$.
    Let $\nu \colon X/A \longrightarrow B$ be the quotient morphism.
    The morphism $\overline{\f}$ descends to $\id$ on $B$ and thus 
    by \cref{eq:sum-of-fpii}, $\nu(\pi(a)) \in B(K)_{\rm tors}$.
    Note that we have $B = X/C$ where $C = A + (\f -\id)(X)$ 
    is an abelian subvariety.
    By \cite{Len96}, the isomorphic classes of the quotient of $X$ by an abelian subvariety are finite.
    Thus there is $m_0 \in \Z_{\geq 1}$ depending only on $X$ such that
    $m_0 \nu(\pi(a)) = 0$.
    Thus, we have the commutative diagram
    \begin{equation}
        \begin{tikzcd}
            X \arrow[rr,"f^{m_0}"] \arrow[dr,"\mu",swap] & & X 
            \arrow[ld, "\nu \circ \pi =: \mu"] \\
            & B &
        \end{tikzcd}
    \end{equation}
    where $\mu$ is the quotient morphism of $X$ by $C$.
    Now let $x \in X(K)$ be such that $f^{s}(x) \in Y$ for some $s \in \Z_{\geq 0}$.
    Then $(f^{m_0})^{s}(x) \in Y$ for some $s \in \Z_{\geq 0}$ as well.
    Note that $f^{m_0}$ induces
    \begin{align}
        f^{m_0} \colon \mu^{-1}(\mu(x)) = x+C \longrightarrow x+C.
    \end{align}
    Note also that $x+C$ can be equipped with an abelian variety structure 
    so that $x$ is the identity element so that it is isomorphic to $C$
    as abelian varieties.
    As $\dim B \geq 1$, we have $\dim C < \dim X$, 
    we can apply the induction hypothesis to
    $f^{m_0}|_{x+C}$ and $Y \cap (x+C)$.
    Then there is $s_0$ depending only on $C$
    such that $(f^{m_0})^{s_0}(x) \in Y$.
    Since $C$ is an abelian subvariety of $X$, by \cite{Len96} again, 
    we can take $s_0$ depending only on $X$.

    {\bf Case 2}.
    Suppose $B=0$.
    Then $\overline{\f} - \id$ is an isogeny and hence $H=0$.
    In particular, we have $\pi(Y_1) = 0$.
    Moreover, $P_{\overline{\f}}$ has no cyclotomic factor because 
    if it has, it must be the power of $t-1$.

    By \cref{eq:sum-of-fpii}, $\overline{f}^m = \overline{\f}^m$.
    Noting $\pi(Y_2) = \overline{f}(\pi(Y_1)) = \pi(a)$, we have
    $\overline{\f}^m(\pi(a))= \pi(a)$.
    Hence $\pi(a) \in (X/A)(K)_{\rm tors}$.

    Now let $x \in X(K)$ be such that $f^{s}(x) \in Y$ for some $s \in \Z_{\geq 0}$.
    Then $f^s(x) \in Y_1$ for some $s \in \Z_{\geq 1}$ as well.
    Then $\overline{f}^s(\pi(x)) = 0$ and this implies 
    \begin{align}
        \overline{\f}^s(\pi(x)) = -\sum_{i=0}^{s-1} \overline{\f}^i(\pi(a)).
    \end{align}
    The right-hand side is a torsion point, and thus $\pi(x) \in (X/A)(K)_{\rm tors}$.
    Since $\pi(a)$ is a torsion point, 
    we have $O_{\overline{f}}(\pi(x)) \subset (X/A)(K)_{\rm tors}$.
    Again by \cite{Len96}, there is $m_0 \in \Z_{\geq 1}$ depending only on $X$
    such that $\# X/A)(K)_{\rm tors} \leq m_0$.
    Then we have $\overline{f}^{m_0}(\pi(x)) \in \pi(Y)$ and therefore
    $f^{m_0}(x) \in \pi^{-1}(\pi(Y)) = Y + A = Y$.
    This completes the proof.
\end{proof}

In the proof of \cref{thm:uniform-PIQ-ab-var}, we use several finiteness 
properties of abelian varieties, namely finiteness of the set of torsion points,
finiteness of Mordell-Weil rank, and finiteness of abelian subvarieties \cite{Len96}.
As for the first one, there is a following strong uniformity conjecture. 

\begin{conjecture}\label{conj:general_tors_conj}
    Let $g \in \Z_{\geq 1}$ and $d \in \Z_{\geq 1}$.
    Then there exists $C \geq 1$ depending only on $g$ and $d$ such that
    for all abelian variety $X$ of dimension $g$ defined over a number field 
    $K$ with $[K : \Q] \leq d$, we have 
    \begin{align}
        \# X(K)_{\rm tors} \leq C.
    \end{align}
\end{conjecture}

If we assume this conjecture, we have the following.

\begin{proposition}\label{prop:uniform_PIQ_tors_conj}
    Assume \cref{conj:general_tors_conj}.
    Let $g \in \Z_{\geq 1}$ and $d \in \Z_{\geq 1}$.
    Then there exists $s_0 \geq 0$ depending only on $g$ and $d$ with the following property.
    Let $X$ be an abelian variety of dimension $g$ defined over a number field 
    $K$ with $[K : \Q] \leq d$.
    Let $f \colon X \longrightarrow X$ be a surjective morphism, and 
    let $Y \subset X$ be an $f$-invariant geometrically irreducible closed subvariety.
    Then
    \begin{align}
        ( f^{-s_0 -1}(Y) \smallsetminus f^{-s_0}(Y) )(K) = \emptyset.
    \end{align}
\end{proposition}

\begin{proof}
    Let $X, f, Y$ be as in the statement.
    We may assume there exists at least one $K$-rational point $y \in Y(K)$.
    Let $A$ be the identity component of the stabilizer $\Stab(Y)$ of $Y$ in $X$.
    Let us write $f = T_a \circ \f$ with $a \in X(K)$ and a group endomorphism 
    $\f \colon X \longrightarrow X$.
    Then we have $\f(A) = A$.
    Indeed, 
    \begin{align}
        Y + \f(A) = (a + \f(Y)) + \f(A) = a + \f(Y + A) = a + \f(Y) = Y
    \end{align}
    and this implies $\f(A) = A$.

    Consider the following commutative diagram:
    \begin{equation}
        \begin{tikzcd}
            X \arrow[r,"f"] \arrow[d,"\pi",swap] & X \arrow[d,"\pi"] \\
            X/A \arrow[r,"\overline{f}:= T_{\pi(a)}\circ \overline{\f} " {yshift=-1ex},swap] & X/A
        \end{tikzcd}
    \end{equation}
    where $\pi$ is the quotient morphism and $\overline{\f}$ is the group morphism
    induced by $\f$.
    Then $\Stab(\pi(Y)_{\overline{K}}) = \Stab(\pi(Y))_{\overline{K}}$ is finite.
    Since $\pi(Y)$ is geometrically irreducible, by \cite[(3.7) Theorem]{Mo85},
    $\pi(Y)_{\overline{K}}$ is a variety of general type.
    Thus there is $n \in \Z_{\geq 1}$ such that $\overline{f}|_{\pi(Y)}^n = \id_{\pi(Y)}$.

    Let $P_{\overline{\f}}$ be the minimal polynomial of $\overline{\f}$.
    Let us factor $P_{\overline{\f}}(t) = Q(t) R(t)$ such that
    $Q$ is a product of cyclotomic polynomials and $R$ has no cyclotomic factors.
    Then set $B = R(\overline{\f})(X/A)$.
    Since $B$ is $\overline{\f}$-invariant and thus we can form the following commutative diagram:
    \begin{equation}
        \begin{tikzcd}
            X/A \arrow[r,"\overline{f}"] \arrow[d,"\nu",swap] & X/A \arrow[d,"\nu"]\\
            (X/A)/B \arrow[r,"h:= T_{\nu(\pi(a))} \circ \psi"{yshift=-1ex},swap] & (X/A)/B
        \end{tikzcd}
    \end{equation}
    where $\nu$ is the quotient morphism and $\psi$ is the morphism induced by $\overline{\f}$.
    Then, first note that the minimal polynomial $P_{\psi}(t)$ of $\psi$ is equal to $R(t)$.
    In particular, $\psi^n - \id$ is an isogeny.
    As $\overline{f}|_{\pi(Y)}^n = \id_{\pi(Y)}$, we have $h^n = \id$ on $\nu(\pi(Y))$.
    This implies 
    \begin{align}
        (\psi^n - \id)(\nu(\pi(Y))) = - \sum_{i=0}^{n-1}\psi^i(\nu(\pi(a))).
    \end{align}
    Thus we get $\dim \nu(\pi(Y))=0$.
    Since $Y$ is irreducible, so is $\nu(\pi(Y))$ and hence $\nu(\pi(Y)) = \{ \nu(\pi(y))\}$.
    In particular, we have 
    \begin{align}
        &h(\nu(\pi(y))) = \nu(\pi(y)) \\
        &\pi(Y) \subset \pi(y) + B.
    \end{align}
    Moreover $\pi(y) + B$ is invariant under $\overline{f}$ and the induced self-morphism
    on $\pi(y) + B$ is an isomorphism.
    Indeed, we have the following commutative diagram
    \begin{equation}
        \begin{tikzcd}
            \pi(y) + B \arrow[r,"\overline{f}"]  & \pi(y) + B \\
            B \arrow[u,"T_{\pi(y)}"] \arrow[r,"\overline{\f}",swap] & B \arrow[u,"T_{\overline{\f}(\pi(y))+\pi(a)}",swap]
        \end{tikzcd}
    \end{equation}
    and the bottom arrow $\overline{\f}$ action on $B$ is an isomorphism because 
    its minimal polynomial is $Q$.

    Now let $x \in X(K)$ and suppose $f^s(x) \in Y$ for some $s \in \Z_{\geq 0}$.
    Then we have $h^s(\nu(\pi(x))) = \nu(\pi(y))$.
    Let us consider the following set
    \begin{align}
        S = \{ z \in ((X/A)/B)(K) \mid \text{$h^t(z) = \nu(\pi(y))$ for some $t \in \Z_{\geq 0}$}  \}.
    \end{align}
    For any $z_1, z_2 \in S$, we have $\psi^t(z_1 - z_2) = 0$.
    Since $\psi$ is an isogeny, $z_1 - z_2$ is a torsion point and hence 
    $\# S \leq ((X/A)/B)(K)_{\rm tors}$.
    By \cref{conj:general_tors_conj}, there is $s_0 \in \Z_{\geq 1}$
    depending only on $g, d$ such that $\# S \leq s_0$.
    Then we have $h^{s_0}(\nu(\pi(x))) = \nu(\pi(y))$.
    This implies $\overline{f}^{s_0}(\pi(x)) \in \pi(y) + B$.
    Since $\overline{f}^s(\pi(x)) \in \pi(Y)$ and $\overline{f}|_{\pi(y) + B}$ is an isomorphism
    (as well as $\overline{f}(\pi(Y)) = \pi(Y)$), 
    we must have $\overline{f}^{s_0}(\pi(x)) \in \pi(Y)$.
    Therefore we get
    \begin{align}
        f^{s_0}(x) \in \pi^{-1}(\pi(Y)) = Y + A = Y
    \end{align}
    and we are done.
\end{proof}
    
\begin{remark}
    We do not know if we can get rid of the geometric irreducibility of $Y$.
    In the proof of \cref{thm:uniform-PIQ-ab-var}, we use uniform bound of the period
    of periodic abelian subvarieties to deal with possibly non-irreducible $Y$. 
    Our bound of the periods depends on the Mordell-Weil rank
    of $X(L)$ where $L$ is a finite extension of $K$, which also depends on $X$. 
\end{remark}

\section{Preliminaries on polydisc algebras}\label{sec:Preliminaries on polydisc algebras}

We prepare some preliminary results on polydisc algebras,
especially on prime factorization in the polydisc algebras.
The contents of this section might be well-known to experts, but we include them here for completeness.
For basic materials on Tate algebras or strictly affinoid algebras,
see for example \cite{BGR_NA_analysis,Menares_Tate_alg}.

Let $(K,|\ |)$ be a complete Non-archimedean valued field.
For $r = (r_1, \dots, r_n) \in \R_{>0}^n$, the polydisc algebra with radii $r$ is 
\begin{align}
    K\{ r^{-1}T \} &= K\{ r_1^{-1}T_1, \dots, r_n^{-1}T_n \}\\
    & = \biggl\{ \sum_{I \in \Z_{\geq 0}^n} a_I T^I  \in K\lb T \rb \mathrel{\bigg|}  \lim_{I \to \infty}|a_I|r^I =0 \biggr\}.
\end{align}
Here we use the standard multi index notation: for $I = (i_1, \dots, i_n)$,
$T^I := T_1^{i_1}\cdots T_n^{i_n}$ and $r^I = r_1^{i_1}\cdots r_n^{i_n}$.
The Gauss norm on $K\{ r^{-1}T  \}$ is
\begin{align}
    \biggl\|\sum_{I \in \Z_{\geq 0}^n} a_I T^I\biggr\|_{ K\{ r^{-1}T  \}}
    = \max\left\{ |a_I|r^I \ \middle|\  I \in \Z_{\geq 0}^n \right\}.
\end{align}
We drop the subscript and simply write $\|\ \|$ when there should be no confusion.
For $I = (i_1,\dots, i_n)\in \Z_{\geq 0}^n$, we set $|I|= i_1 + \cdots + i_n$.

\begin{definition}
    For $f \in K\{ r^{-1}T \} \smallsetminus \{0\}$, we define 
    \begin{align}
        \ord(f)=\ord_{K\{ r^{-1}T \}}(f) 
        := \max \left\{ |I| \ \middle|\  |a_I|r^I = \|f\| \right\}.
    \end{align}
    When $f=0$, we set $\ord_{K\{ r^{-1}T \}}(0) = \infty$.
\end{definition}

\begin{lemma}\label{lem:property-of-ord}
    Let $K$ be a complete Non-archimedean valued field and $r = (r_1, \dots, r_n) \in \R_{>0}^n$.
    \begin{enumerate}
        \item For $f,g \in  K\{ r^{-1}T \}$, we have $\ord(fg) = \ord(f) + \ord(g)$.
        \item For $f \in  K\{ r^{-1}T \}$, $f$ is unit if and only if $\ord(f) = 0$.
        \item For any finite extension $L$ of $K$ equipped with the uniquely extended absolute value
        and $f \in K\{ r^{-1}T \}$, we have
        \begin{align}
            \ord_{L\{ r^{-1}T \}}(f) = \ord_{K\{ r^{-1}T \}}(f).
        \end{align}
        \item For $f \in K\{ r^{-1}T \} \smallsetminus \{0\}$, the following set is finite:
        \begin{align}
            \I(f) 
            := \left\{ I \in \Z_{\geq 0}^n 
            \ \middle|\  \exists s \in \prod_{i=1}^n (0,r_i], |a_I|s^I = \|f\|_{K\{ s^{-1}T \}} \right\}.
        \end{align}
        In particular, the set of orders
        \begin{align}
            \left\{ \ord_{K\{ s^{-1}T \}}(f) \ \middle|\ s \in \prod_{i=1}^n (0,r_i] \right\}
        \end{align}
        is finite.
    \end{enumerate}
\end{lemma}

\begin{proof}
    (1)
    Let us write $f = \sum_{k \geq 0}f_k$, $g = \sum_{l \geq 0}g_l$
    where $f_k, g_l$ are homogeneous part of degree $k,l$ respectively.
    Let $d = \ord (f), e = \ord (g)$.
    Then, by the standard argument, we have
    \begin{align}
        \left\|\sum_{i=0}^{d+e} f_i g_{d+e -i}\right\| &= \left\| f_d g_e\right\| = \left\|fg\right\|,\\
        \left\|\sum_{i=0}^{m} f_i g_{m -i}\right\| &< \left\| f_d g_e\right\|\quad (\text{for } m > d + e)
    \end{align}
    and thus $\ord(fg) = d + e$.

    (2)
    If $f$ is a unit, then $0=\ord(1)=\ord(f f^{-1}) = \ord(f) + \ord(f^{-1})$.
    This implies $\ord(f) = 0$.
    Suppose $\ord(f) = 0$.
    Then $f(0) \neq 0$ and we can write
    \begin{align}
        f = f(0) \left( 1 + g\right)
    \end{align}
    with $g$ such that $\|g\| < 1$.
    Thus, $f$ is a unit.

    (3) This is obvious from the definition.

    (4)
    For $I = (i_1, \dots, i_n) \in \Z_{\geq 0}^n$, let 
    \begin{align}
        C(I) = \{ (j_1,\dots, j_n) \in \Z_{\geq 0}^n \mid \text{$j_k \geq i_k$ for $k=1,\dots,n$} \}.
    \end{align}
    Note that for any subset $\I \subset \Z_{\geq 0}^n$, there are finitely many elements $I_1, \dots, I_t$ such that 
    \begin{align}
        \bigcup_{I \in \I} C(I) = \bigcup_{k=1}^t C(I_k).
    \end{align}
    To see this, note that $C(I)$ is exactly the set of exponents of monomials
    contained in the ideal generated by $x_1^{i_1}\cdots x_n^{i_n}$ 
    in the polynomial ring $\Q[x_1,\dots,x_n]$.
    The left-hand side is the set of exponents of monomials contained in 
    the ideal generated by $(x^I)_{I \in \I}$, but the Hilbert basis theorem says
    this ideal is generated by finitely many $x^I$'s.

    Now, we go back to our situation. 
    We prove (4) by induction on $n$.
    The statement is clear when $n = 1$.
    Suppose $n \geq 2$ and $\#\I(f)= \infty$.
    By the above discussion, there are $I_1,\dots, I_t \in \I(f)$ such that
    \begin{align}
        \I(f) \subset \bigcup_{I \in \I(f)} C(I) = \bigcup_{k=1}^t C(I_k).
    \end{align}
    Since $\#\I(f)= \infty$, there is $I_k$ such that
    \begin{align}
        \# \I(f) \cap C(I_k) = \infty.
    \end{align}
    Define $J_1 = I_k$.
    Next,
    \begin{align}
         \bigcup_{I \in \I(f) \cap C(J_1) \smallsetminus \{J_1\}}C(I)
    \end{align}
    is also a finite union, and thus there is 
    $J_2 \in \I(f) \cap C(J_1) \smallsetminus \{J_1\}$ 
    such that $\# \left(\I(f) \cap C(J_1) \smallsetminus \{J_1\}\right) \cap C(J_2)=\infty $.
    Repeat this and we get a sequence $(J_k)_{k=1}^\infty$ such that
    \begin{align}\label{seq:Jk}
        &J_k \in \I(f) \ \text{for $k = 1,2,\dots$}\\
        &C(J_1) \supsetneq C(J_2) \supsetneq \cdots.
    \end{align}
    Now let us write $f = \sum_I a_I T^I$ and take $I_0 \in \Z_{\geq 0}^n$ such that 
    \begin{align}
        |a_{I_0}|r^{I_0} = \|f\|_{K\{r^{-1}T\}}\ \text{and}\ 
        |I_0| = \ord_{K\{r^{-1}T\}}(f). 
    \end{align}
    Then we have $\I(f) \cap C(I_0) = \{I_0\}$.
    Indeed, if $J \in C(I_0) \smallsetminus \{I_0\}$ and $s \in \prod_{i=1}^n (0,r_i]$, then
    \begin{align}
        |a_J|s^J = |a_J|r^J \frac{s^J}{r^J} < |a_{I_0}|r^{I_0} \frac{s^J}{r^J}
        = |a_{I_0}|s^{I_0}\frac{r^{I_0}}{s^{I_0}} \frac{s^J}{r^J} \leq |a_{I_0}|s^{I_0}.
    \end{align}
    By \cref{seq:Jk}, we have $J_k \not\in C(I_0)$ for all $k$.
    By replacing it with a subsequence, we may assume that one of the coordinates of 
    $J_k$ are constant. By changing coordinates, we may assume the last coordinates are 
    constant. Let us write $J_k = (J_k', j)$ where $J_k' \in \Z_{\geq 0}^{n-1}$ and
    $j \in \Z_{\geq 0}$.
    Consider 
    \begin{align}
        g = \sum_{k=1}^\infty a_{J_k}T^{J_k} \frac{1}{T_n^j}
        \in K\{r_1^{-1}T_1, \dots, r_{n-1}^{-1}T_{n-1}\}.
    \end{align}
    Since $J_k \in \I(f)$, we have $J_k' \in \I(g)$.
    This implies that $\I(g)$ is infinite, which contradicts the induction hypothesis.
\end{proof}

\begin{corollary}\label{cor:prime-fac-bdd}
    Let $K$ be a complete Non-archimedean valued field 
    and $r = (r_1, \dots, r_n) \in \R_{>0}^n$.
    Let $f \in K\{ r^{-1}T \} \smallsetminus \{0\}$.
    Then there is $C > 0$ with the following property: 
    for all finite extension $L$ of $K$ with the uniquely extended absolute value and
    for all $s = (s_1,\dots , s_n) \in \prod_{i = 1}^n (0,r_i]$ with $s_i \in |L|$ for $i=1,\dots n$, we have
    \begin{align}
        \left(\txt{the number of prime factors counted with multiplicity \\ of $f$ as an element of $L\{s^{-1}T\}$}\right) \leq C.
    \end{align}
    Here note that $L\{s^{-1}T\}$ is UFD, because $s_i \in |L|$
    and hence $L\{s^{-1}T\} \simeq L\{T\}$.
    (cf.\ \cite[(5.2.6) Theorem 1]{BGR_NA_analysis}.)
\end{corollary}

\begin{proof}
    Let $L$ and $s$ as in the statement.
    Let $f = P_1 \dots P_l$ be the prime factorization in $L\{s^{-1}T\}$.
    Then, by \cref{lem:property-of-ord} (1), (2), and (3), we obtain
    \begin{align}
        \ord_{K\{s^{-1}T\}}(f) = \ord_{L\{s^{-1}T\}}(f)
        = \sum_{i=1}^l \ord_{L\{s^{-1}T\}}(P_i) \geq l.
    \end{align}
    By \cref{lem:property-of-ord} (4), $\ord_{K\{s^{-1}T\}}(f) $ is bounded 
    when $s$ varies, we are done.
\end{proof}

\begin{corollary}\label{cor:prime-fac-fld-ext}
    Let $K$ be a complete Non-archimedean valued field.
    Let $r = (r_1, \dots, r_n) \in \R_{>0}^n$ with
    $r_i \in |K|$ for $i=1,\dots,n$.
    Let $f \in K\{ r^{-1}T \} \smallsetminus \{0\}$.
    Then there is a finite extension $L$ of $K$ such that 
    for all finite extensions $L'$ of $L$, the prime factorization
    of $f$ in $L\{r^{-1}T\}$ is the prime factorization of
    $f$ in $L'\{r^{-1}T\}$ as well.
\end{corollary}

\begin{proof}
    By \cref{cor:prime-fac-bdd}, the number of prime factors of $f$
    as an element of $L\{r^{-1}T\}$ are bounded when $L$ varies finite extensions of $K$.
    Take $L$ such that the number of prime factors attains the maximum.
    Let $f=P_1 \dots P_l$ be the prime factorization in $L\{r^{-1}T\}$.
    For any finite extension $L'$ of $L$, every $P_i$ is not a unit
    in $L'\{r^{-1}T\}$ because $\ord_{L'\{r^{-1}T\}}(P_i) = \ord_{L\{r^{-1}T\}}(P_i) >0$.
    Therefore, $f=P_1 \dots P_l$ must be the prime factorization of $f$
    as an element of $L'\{r^{-1}T\}$.
\end{proof}

{\bf Asymptotic behavior of prime factorization}

Let $K$ be a complete Non-archimedean valued field as before.
Let $r = (r_1, \dots, r_n) \in \R_{>0}^n$ and suppose $r_i \in |K|$ for $i=1,\dots n$.
Fix a sequence of finite extensions of fields
\begin{align}
    K = K_0 \subset K_1 \subset K_2 \subset \cdots
\end{align}
and equip $K_m$ with the uniquely extended absolute values for $m \in \Z_{\geq 1}$.
Fix also a sequence of radii
\begin{align}
    \{ s^{(m)} = (s_1^{(m)}, \dots, s_n^{(m)}) \}_{m=0}^{\infty}
\end{align}
in $\Z_{>0}^n$ such that 
\begin{align}
    &r_j = s_j^{(0)} \geq s_j^{(1)} \geq s_j^{(2)} \geq \cdots \quad (j = 1,2, \ldots, n)\\
    &s_j^{(m)} \in |K_m| \quad (\text{$j = 1,2, \ldots, n$ and  $m\geq 0$}).
\end{align}
We set
\begin{align}
    R_m = K_m \{(s_1^{(m)})^{-1}T_1, \dots, (s_n^{(m)})^{-1}T_n\}.
\end{align}
Note that $R_m$ are UFD.
We have a sequence of inclusions
\begin{align}
    R_0 \xlongrightarrow{i_1} R_1 \xlongrightarrow{i_2} R_2 \xlongrightarrow{i_3} \cdots.
\end{align}
An element $u \in R_m$ for some $m$ is said to be eventually unit if there is some $m' \geq m$
such that $u$ is unit in $R_{m'}$.
In this case, we say $u$ is eventually unit with respect to
$(R_m)_m$.

\begin{lemma}\label{lem:prim-fact-ev-stab}
    Under the setting above, let $f \in R_0\setminus{0}$.
    Let $f_m$ be the image of $f$ in $R_m$.
    Then there is $m_0 \geq 0$ such that
    \begin{align}
        f_{m_0} = u g_1 \cdots g_l
    \end{align}
    for some $u, g_1, \dots, g_l \in R_{m_{0}}$ with the following properties:
    \begin{itemize}
        \item $u$ is eventually unit:
        \item for any $m \geq m_0$, the image of $g_j$ in $R_m$ is of the form
        \begin{align}
            v h
        \end{align}
        where $v \in R_m$ is eventually unit and $h \in R_m$ is a prime element.
    \end{itemize}
\end{lemma}

\begin{proof}
    Define an inverse system 
    \begin{align}
        \cdots S_m \xlongrightarrow{\f_{m-1}} S_{m-1} \longrightarrow \cdots \longrightarrow S_1 \xlongrightarrow{\f_0} S_0
    \end{align}
    as follows.
    Label the prime factors of $f$ as an element of $R_0$.
    If the same prime factor appears multiple times,
    we attach distinct labels to them.
    Let $S_0$ be the set of the labels. Note that if $f$ is unit, then $S_0 = \emptyset$.
    Next, for each prime factor $P$ of $f$ in $R_0$,
    we similarly label the prime factors of $P$ as an element of $R_1$.
    Then, set $S_1$ as their disjoint union varying $P$.
    Define $\f_0 \colon S_1 \longrightarrow S_0$ so that
    the labels corresponding to factors of $P$ are mapped to the label corresponding to $P$.
    Repeat this, and we get the inverse system.
    For each $i \in S_m$, let us write the corresponding prime factor 
    of $f$ in $R_m$ as $P_{m,i}$.
    In particular we have $f_m = u\prod_{i \in S_m}P_{m,i}$ for some $u \in R_m^{\times}$.
    
    First note that $S_m = \emptyset$ for some $m$ if and only if $f$ is eventually unit.
    In this case, the statement is clear.
    Suppose $S_m \neq \emptyset$ for all $m \geq 0$.
    Then 
    \begin{align}
        S := \varprojlim_{m\geq 0} S_m \neq \emptyset.
    \end{align}
    By \cref{cor:prime-fac-bdd}, there is $C > 0$ such that $\# S_m \leq C$ for all $m$.
    This implies $\# S \leq C$. Let $\# S = l$.
    Let us write 
    \begin{align}\label{eq:members-of-S}
        S = \left\{ (i_1(m))_m, \dots, (i_l(m))_m \right\}.
    \end{align}
    Then there is $m_0$ such that 
    \begin{align}
        i_1(m_0),\dots, i_l(m_0)
    \end{align}
    are distinct.
    Then, we can write 
    \begin{align}
        f_{m_0} = u \prod_{i \in S_{m_0} \smallsetminus \{i_1(m_0),\dots, i_l(m_0) \}} P_{m_0,i}
        \prod_{j=1}^l P_{m_0, i_j(m_0)}
    \end{align}
    for some $u \in R_{m_0}^{\times}$.
    
    By \cref{eq:members-of-S}, for each 
    $i \in S_{m_0} \smallsetminus \{i_1(m_0),\dots, i_l(m_0) \}$
    there is $m \geq m_0$ such that
    \begin{align}
        (S_m \longrightarrow S_{m_0})^{-1}(i) = \emptyset.
    \end{align}
    Indeed, if there is no such $m$, an element of
    \begin{align}
        \varprojlim_{m \geq m_0} (S_m \longrightarrow S_{m_0})^{-1}(i)
    \end{align}
    produces an element of $S$ different from any of $(i_j(m))_m$.
    In particular, we proved that $P_{m_0, i}$ is eventually a unit.
    
    Finally, for every $m \geq m_0$, we can write
    \begin{align}
        P_{m_0, i_j(m_0)} = v P_{m, i_j(m)} 
        \prod_{\substack{k \in S_m \\ k \neq i_j(m)\\ (S_m \longrightarrow S_{m_0})(k)=i_j(m_0)}}P_{m,k}
    \end{align}
    in $R_m$ for some $v \in R_m^{\times}$.
    As before, the last factor is eventually a unit.
    Thus, we are done.
\end{proof}

{\bf Question.}
Is it possible to get rid of ``eventually"?
That is, is there such $m_0$ with
$u$ being unit in $R_{m_0}$ and $g_j$ being prime in $R_m$ for all $m \geq m_0$?

\section{\texorpdfstring{$p$}{p}-adic uniformization}\label{sec:padic uniformization}

In this section, we recall Rivera-Letelier's $p$-adic uniformization theorem and some basic facts on formal power series.

The following is well-known, and the proof is elementary.

\begin{lemma}
    Let $R$ be a ring.
    Let $f = \sum_{n \geq 1}a_n z^n \in R\lb z \rb$ be a formal power series.
    If $a_1 \in R^{\times}$, then there is unique $g \in zR\lb z \rb$ such that
    $f(g(z)) = g(f(z)) = z$.
\end{lemma}

Convergent power series that admit compositional inverse induces isomorphism between small discs around $0$.
The following is an algebraic expression of this fact.

\begin{lemma}\label{lem:isom-of-small-discs}
    Let $(K, |\ |)$ be a complete non-archimedean field and $r_0, s_0 >0$.
    Let $u \in K\{r_0^{-1}z\}$ and $v \in K\{s_0^{-1}z\}$ with $u(0)=v(0)=0$. 
    Suppose $u(v(z)) = v(u(z)) = z$ in $K\lb z \rb$.
    Then there is $r_1 \in (0,r_0 ] $ with the following properties.
    For all $r \in (0,r_1]$, 
    the morphism
    \begin{align}
        K\{ (|u'(0)|r)^{-1} z \} \longrightarrow K\{ r^{-1}z \}\ ; z \mapsto u(z)
    \end{align}
    is an isometric isomorphism.
    Moreover, $\|v\|_{|u'(0)|r} = r$ and 
    \begin{align}
          K\{ r^{-1}z \} \longrightarrow K\{ (|u'(0)|r)^{-1} z \} \ ; z \mapsto v(z)
    \end{align}
    is the inverse of the above morphism.
    Here $\|\ \|_r$ denotes the Gauss norm on $K\{r^{-1}z\}$.
\end{lemma}

\begin{proof}
    Let us write 
    \begin{align}
        &u(z) = a z + \widetilde{u}(z)\\
        &v(z) = b z + \widetilde{v}(z)
    \end{align}
    for some $a,b \in K \smallsetminus \{0\}$ and  
    $\widetilde{u}(z),\widetilde{v}(z) \in z^2 K \lb z \rb$.
    Then $u'(0) = a$.
    We have $ab = 1$ because $u(v(z)) = z$.
    
    Now, there is $r_1 \in (0, r_0]$ such that
    $\|u\|_r = |a|r$ for all $r \in (0,r_1]$.
    If we further make $r_1$ small, then we also have
    $\|v\|_{\|u\|_r} = \|v\|_{|a|r} = |b||a|r = r$ for all $r \in (0,r_1]$.
    These imply the morphisms in the statements are well-defined and contractive.
    Since $u(v(z)) = v(u(z)) = z$, they are inverse maps of each other.
\end{proof}

Now, we state the $p$-adic uniformization in the form we use later.

\begin{theorem}\label{thm:p-adic-unif}
    Let $r>0$ and $K$ be a finite extension of $\Q_p$
    equipped with the multiplicative absolute value $|\ |$ extending $|\ |_p$.
    Let $f = \sum_{n \geq 1} a_n z^n \in K\{ r^{-1} z\}$
    \begin{enumerate}
        \item Suppose $0 < |a_1| < 1$.
        Then there is $r' \in (0, r]$, $s > 0$, and isometric isomorphism 
        $\f \colon K\{s^{-1}z \} \longrightarrow K\{ {r'}^{-1}z\}$ such that
        \begin{itemize}
            \item $\f(z)(0) = 0$
            \item $\|f\|_{K\{{r'}^{-1}z\}} \leq r'$
            \item the following diagram commutes.
            \begin{center}
            \begin{tikzcd}
            z \arrow[r,mapsto] & f(z)  \\[-2em]
            K\{ {r'}^{-1} z\}  \arrow[r] & K\{ {r'}^{-1} z\} \\
            K\{ s^{-1} z\}  \arrow[r] \arrow[u,"\f"] & K\{ s^{-1} z\} \arrow[u,"\f",swap]\\[-2em]
            z \ar[r,mapsto] & a_1 z&
            \end{tikzcd}
            \end{center}
        \end{itemize}

        \item Suppose $a_1=\cdots = a_{d-1}=0$ and $a_d \neq 0$ for some $d \geq 2$.
        Let $L$ be a finite extension of $K$ containing a $(d-1)$-th root of $a_d$.
        Then there is $r' \in (0, r]$, $s > 0$, and isometric isomorphism 
        $\f \colon L\{s^{-1}z \} \longrightarrow L\{ {r'}^{-1}z\}$ such that
        \begin{itemize}
            \item $\f(z)(0) = 0$
            \item $\|f\|_{L\{{r'}^{-1}z\}} \leq r'$
            \item the following diagram commutes.
            \begin{center}
            \begin{tikzcd}
            z \arrow[r,mapsto] & f(z)  \\[-2em]
            L\{ {r'}^{-1} z\}  \arrow[r] & L\{ {r'}^{-1} z\} \\
            L\{ s^{-1} z\}  \arrow[r] \arrow[u,"\f"] & L\{ s^{-1} z\} \arrow[u,"\f",swap]\\[-2em]
            z \ar[r,mapsto] &  z^d&
            \end{tikzcd}
            \end{center}
        \end{itemize}
    \end{enumerate}
\end{theorem}
\begin{proof}
    See the proof of \cite[Proposition 3.3]{RL03}.
    See also \cite[Theorem 1]{HY83}.
\end{proof}

\section{Preimages Question on \texorpdfstring{$\P^1 \times \P^1$}{P1P1}}\label{sec:piq P1P1}

In this section, the $p$-adic absolute value $|\ |_p$ on $\Q_p$
is simply denoted by $|\ |$.
All algebraic extensions of $\Q_p$ as well as $\C_p$ are equipped with
the multiplicative absolute values extending $|\ |_p$
and they are denoted by $|\ |$.

The following is the main theorem in this section.

\begin{theorem}\label{thm:piq_P1P1}
    Let $K$ be a finitely generated field over $\Q$. 
    Let $\f \colon \P^1_K \times_K \P^1_K \longrightarrow \P^1_K \times_K \P^1_K$
    be a surjective morphism over $K$.
    Let $Y \subset \P^1_K \times_K \P^1_K$ be a $\f$-invariant closed subscheme.
    Then there is $s_0 \geq 0$ such that for all $s \geq s_0$, we have
    \begin{align}
        \left(\f^{-s-1}(Y) \smallsetminus \f^{-s}(Y) \right)(K) = \emptyset.
    \end{align}
\end{theorem}
\begin{proof}
     By replacing $\f$ with $\f^2$, we may write $\f = f \times g$ for some
     finite surjective morphisms $f,g \colon \P^1_K \longrightarrow \P^1_K$.
     Take a subring $R \subset K$ which is finitely generated over $\Z$ so that
     $K = \Frac R$ and there are 
     \begin{itemize}
         \item finite surjective morphisms $f_R,g_R \colon \P^1_R \longrightarrow \P^1_R$;
         \item closed subscheme $Y_R \subset \P^1_R \times_R \P^1_R$
     \end{itemize}
     with the following properties.
     The morphisms $f,g$ are base change of $f_R$ and $g_R$, $Y$ is the base change of $Y_R$,
     and $Y_R$ is $f_R \times g_R$-invariant.
     By \cite[Proposition 2.5.3.1]{BGT16} (cf.\ \cite[Proposition 3.6]{BMS23}),
     there is a prime number $p$ and an embedding $\s \colon K \longrightarrow \Q_p$ such that
     \begin{itemize}
         \item $\s(R) \subset \Z_p$
         \item $p$ does not divide any ramification indices of $f$ and $g$ considered as 
         self-morphisms on $\P^1_{\overline{K}}$.
     \end{itemize}
    Then let us write the base change of $f_R, g_R$, and $Y_R$ to $\Z_p$ (resp. $\Q_p$)
    via $\s$ as $f_{\Z_p}, g_{\Z_p}$, and $Y_{\Z_p}$ 
    (resp. $f_{\Q_p}, g_{\Q_p}$, and $Y_{\Q_p}$), cf.\ the following diagram.
    \begin{center}
    \begin{tikzpicture}[commutative diagrams/every diagram,scale=0.8,transform shape]
        \node (P0) at (3,2.5) {$\P^1_{\Z_p} \times_{\Z_p} \P^1_{\Z_p}$};
        \node (P01) at (3,2.6) {\phantom{P}};
        \node (P1) at (8,2.5) {$\P^1_{\Q_p} \times_{\Q_p} \P^1_{\Q_p}$};
        \node (P11) at (8,2.6) {\phantom{P}};
        \node (P2) at (0,1) {$\P^1_R \times_R \P^1_R$};
        \node (P21) at (-0.75,1) {\phantom{P}};
        \node (P3) at (5,1) {$\P^1_K \times_K \P^1_K$};
        \node (P31) at (5.75,0.9) {\phantom{P}};
        \node (P4) at (3,-0.5) {$\Spec \Z_p$};
        \node (P5) at (8,-0.5) {$\Spec \Q_p$};
        \node (P6) at (0,-2) {$\Spec R$};
        \node (P7) at (5,-2) {$\Spec K$};
        \node (M0) at (-2.1,1) {\text{\scriptsize $f_R \times g_R$}};
        \node (M1) at (3,3.5) {\text{\scriptsize $f_{\Z_p} \times g_{\Z_p}$}};
        \node (M2) at (8,3.5) {\text{\scriptsize $f_{\Q_p} \times g_{\Q_p}$}};
        \node (M3) at (6.8,0.2) {\text{\scriptsize $f_K \times g_K$}};

        \path[commutative diagrams/.cd, every arrow, every label]
        (P0) edge (P2)
        (P0) edge (P4)
        (P1) edge (P0)
        (P1) edge (P3)
        (P1) edge (P5)
        (P2) edge (P6)
        (P3) edge (P2)
        (P3) edge (P7)
        (P4) edge (P6)
        (P5) edge (P4)
        (P5) edge (P7)
        (P7) edge (P6)
        (P21) edge[out=-150,in=150,looseness=4] (P21)
        (P01) edge[out=120,in=60,looseness=4] (P01)
        (P11) edge[out=120,in=60,looseness=4] (P11)
        (P31) edge[out=-70,in=0,looseness=4] (P31);    
    \end{tikzpicture}
    \end{center}
    Then, we have the following commutative diagram.
    \begin{equation}
        \begin{tikzcd}
            (\P^1 \times \P^1)(K) \arrow[r,hookrightarrow] \arrow[d,"f \times g"] & (\P^1 \times \P^1)(\Q_p) \arrow[r,"\sim"] \arrow[d,"f_{\Q_p} \times g_{\Q_p}"] & (\P^1 \times \P^1)(\Z_p) \arrow[d,"f_{\Z_p} \times g_{\Z_p}"]\\
            (\P^1 \times \P^1)(K) \arrow[r,hookrightarrow] & (\P^1 \times \P^1)(\Q_p) \arrow[r,"\sim"] & (\P^1 \times \P^1)(\Z_p)
        \end{tikzcd}
    \end{equation}
    We also have
    \begin{align}
        &Y_{\Q_p}(\Q_p) \cap (\P^1 \times \P^1)(K)  = Y(K)\\
        &Y_{\Q_p}(\Q_p) = Y_{\Z_p}(\Z_p)\ \text{via the above identification.}
    \end{align}
    Therefore, PIQ for $f \times g$ and $Y$ follows from that of
    $f_{\Z_p} \times g_{\Z_p} \colon (\P^1 \times \P^1)(\Z_p) \longrightarrow (\P^1 \times \P^1)(\Z_p) $ and $Y_{\Z_p}(\Z_p)$.
    By the choice of $p$, $p$ does not divide the ramification indices of 
    $f_{\Q_p}$ and $g_{\Q_p}$ considered as self-morphisms on $\P^1_{\overline{\Q}_p}$.
    Thus, the theorem follows from the following \cref{prop:PIQ-P1P1-Zp}.
\end{proof}

\begin{proposition}\label{prop:PIQ-P1P1-Zp}
    Let $f, g \colon \P^1_{\Z_p} \longrightarrow \P^1_{\Z_p}$ be finite surjective morphism over $\Z_p$.
    Suppose $p$ is coprime with every ramification index of $f$ and $g$
    considered as self-morphisms on $\P^1_{\overline{\Q}_p}$.
    Let $Y \subset \P^1_{\Z_p} \times_{\Z_p} \P^1_{\Z_p}$ be an $f \times g$-invariant closed subscheme.
    Then there is $s_0 \geq 0$ such that 
    if $x \in (\P^1 \times \P^1)(\Z_p)$ satisfies $f^s(x) \in Y(\Z_p)$ for some $s \geq 0$,
    then we have $f^{s_0}(x) \in Y(\Z_p)$.
\end{proposition}

\begin{proof}
    Since we are interested in $\Z_p$-points, we may assume $Y$ is reduced.
    By replacing $f, g$ with some iterates, we may assume $Y$ is irreducible.
    If either $Y = \P^1_{\Z_p} \times_{\Z_p} \P^1_{\Z_p}$ or $Y(\Z_p) = \emptyset$,
    there is nothing to prove.
    Thus, we may assume $\dim Y = 1$ or $2$.
    When $\dim Y = 1$, by replacing $f, g$ with some iterates again,
    we may assume $Y$ is a single $\Z_p$-point $\Spec \Z_p \hookrightarrow \P^1_{\Z_p} \times_{\Z_p} \P^1_{\Z_p}$.
    When $\dim Y = 2$, $Y$ is locally defined by a principal ideal.

    In this proof, we use the following terminology.
    For a triple $(X, f, Y)$ consisting of a set $X$, self-map $f \colon X \longrightarrow X$,
    and subset $Y \subset X$ with $f(Y) \subset Y$, we say PIQ is true for $(X,f,Y)$ if the following holds:
    there is $s_0$ such that if $x \in X$ satisfies $f^s(x) \in Y$ for some $s \geq 0$,
    then $f^{s_0}(x) \in Y$.

    {\bf \fbox{Step 1.}}
    For a morphism $f$ over $\Z_p$ or a $\Z_p$-point $x$,
    their reduction mod $p$ are denoted by $\overline{f}$ and $\overline{x}$.

    Since $\P^1(\F_p) \times \P^1(\F_p)$ is a finite set,
    by replacing $f,g$ with iterates, we may assume $\overline{(f \times g)} = \overline{(f \times g)}^2$
    as self-maps on $\P^1(\F_p) \times \P^1(\F_p)$.
    Let 
    \begin{align}
        \Fix(\overline{f}, \P^1(\F_p)) &= \{ \xi_1,\dots, \xi_{r_1}\},\\ 
        \Fix(\overline{g}, \P^1(\F_p)) &= \{\eta_1,\dots, \eta_{r_2}\}
    \end{align}
    and set
    \begin{align}
        U_{\xi_i} &:= \{ \alpha \in \P^1(\Z_p) \mid \overline{\alpha} = \xi_i\},\\
        V_{\eta_j} &:= \{ \alpha \in \P^1(\Z_p) \mid \overline{\alpha} = \eta_j\}.
    \end{align}
    Then we have $f(U_{\xi_i}) \subset U_{\xi_i}$ and $g(V_{\eta_j}) \subset V_{\eta_j}$.
    Since $\overline{(f \times g)} = \overline{(f \times g)}^2$ on $\P^1(\F_p) \times \P^1(\F_p)$,
    for any point $x \in (\P^1 \times \P^1)(\Z_p)$ we have $(f\times g)(x) \in U_{\xi_i} \times V_{\eta_j}$
    for some $i$ and $j$.
    Therefore, it is enough to show that PIQ is true for 
    \begin{align}
        (f \times g)|_{U_{\xi_i} \times V_{\eta_j}} \colon U_{\xi_i} \times V_{\eta_j} 
        \longrightarrow U_{\xi_i} \times V_{\eta_j}
    \end{align}
    and $Y(\Z_p) \cap (U_{\xi_i} \times V_{\eta_j})$ for every $i$ and $j$.

    Let us fix $i$ and $j$ and simply write $\xi_i$ and $\eta_j$ by $\xi$ and $\eta$, respectively.
    Consider $\xi$ and $\eta$ as closed points of $\P^1_{\Z_p}$ and fix the following
    isomorphisms of $\Z_p$-algebras:
    \begin{align}
        \widehat{\O}_{\P^1_{\Z_p},\xi} \simeq \Z_p \lb z \rb ,\ 
        \widehat{\O}_{\P^1_{\Z_p},\eta} \simeq \Z_p \lb w \rb.
    \end{align}
    By these isomorphisms, we get bijections 
    \begin{align}
        U_{\xi} \simeq p\Z_p ,\ 
        V_{\eta} \simeq p\Z_p  
    \end{align}
    where $a \in p\Z_p$ corresponds to $\Z_p$-algebra homomorphism
    $\Z_p \lb z \rb \longrightarrow \Z_p;\ z \mapsto a$ for example.
    Under these bijections, we get
    \begin{align}
        \xymatrix
        {
        U_{\xi} \times V_{\eta} \ar[d]^{\rotatebox{90}{$\sim$}} \ar[r]^{f\times g}& U_{\xi} \times V_{\eta} \ar[d]^{\rotatebox{90}{$\sim$}} \\
        p\Z_p \times p\Z_p \ar[r]_{F \times G} & p\Z_p \times p\Z_p
        }
    \end{align}
    where $F = (\Z_p \lb z \rb \xlongrightarrow[]{f^*} \Z_p \lb z \rb )(z)$
    and $G = (\Z_p \lb w \rb \xlongrightarrow[]{g^*} \Z_p \lb w \rb )(w)$.
    The image $Z$ of $Y(\Z_p) \cap (U_{\xi} \times V_{\eta})$ to $(p\Z_p)^2$
    is either empty set, singleton, or
    \begin{align}
        \{ (a,b) \in (p\Z_p)^2  \mid \Phi(a,b)=0 \}
    \end{align}
    where $\Phi \in \Z_p \lb z,w \rb$
    is a local defining equation of $Y$ at $(\xi, \eta) \in \P^1_{\Z_p} \times_{\Z_p} \P^1_{\Z_p}$.
    When $Z = \emptyset$, there is nothing to prove.
    When $Z$ is a singleton, it is a fixed point of $F \times G$.
    When $Z$ is the zero set of $\Phi$, we have 
    \begin{align}
        \Phi(z,w) \mid \Phi(F(z), G(w)) \ \text{in $\Z_p \lb z,w \rb$}.
    \end{align}
    Note also that we may assume the constant term of $\Phi$ is zero modulo $p$ since 
    otherwise the zero set of $\Phi$ in $(p\Z_p)^2$ is empty.

    \vspace{3mm}
    
    In summary, we reduced the problem to show PIQ for the map
    \begin{align}\label{map-FcrossG}
        F \times G \colon (p\Z_p)^2 \longrightarrow (p\Z_p)^2
    \end{align}
    and the set $Z \subset (p\Z_p)^2$ where
    \begin{itemize}
        \item $F \in \Z_p \lb z \rb$, $G \in \Z_p \lb w \rb$;
        \item $F,G \neq 0 \mod p$; 
        \item constant terms of $F, G$ are zero modulo $p$;
        \item $Z$ is a fixed point of $F \times G$, or the zero set of $\Phi \in \Z_p \lb z, w \rb \smallsetminus \{0\}$.
        In the latter case, the constant term of $\Phi$ is zero modulo $p$ and
        $\Phi(z,w) \mid \Phi(F(z),G(w))$ in $\Z_p \lb z,w \rb$. 
    \end{itemize}

    {\bf \fbox{Step 2.}}
    Let $F,G$, and $Z$ as above.
    Then $F, G$ define morphisms of strictly $\Q_p$-affinoid algebras
    \begin{align}
        &\Q_p \{ pz \} \longrightarrow \Q_p \{pz\} ; z \mapsto F(z)\\
        &\Q_p \{ pw \} \longrightarrow \Q_p \{pw\} ; w \mapsto G(w)\\
    \end{align}
    because $\|F\|_{\Q_p \{ pz \}} \leq p^{-1}$ and $\|G\|_{\Q_p \{ pw \}} \leq p^{-1}$.
    Their completed tensor product $\f \colon \Q_p\{pz, pw\} \longrightarrow \Q_p\{pz, pw\}$ induces 
    continuous map $\f^* \colon \M(\Q_p\{pz, pw\}) \longrightarrow \M(\Q_p\{pz, pw\})$
    on the Berkovich spectrum. The map \cref{map-FcrossG} is embedded in this map $\f^*$ as maps between type I points:
    \begin{align}
        \xymatrix{
        (p\Z_p)^2 \ar[r]^{F \times G} \ar@{}[d]|{\bigcap} & (p\Z_p)^2 \ar@{}[d]|{\bigcap}\\
        \M(\Q_p\{pz, pw\}) \ar[r]_{\f^*} & \M(\Q_p\{pz, pw\})
        }
    \end{align}
    We also remark that when $Z$ is the zero locus of $\Phi$,
    we have $\Phi \mid \f(\Phi)$ in $\Q_p\{pz, pw\}$.

    Let us write
    \begin{align}
        &F(z) = a_0 + a_1 z + a_2 z^2 + \cdots\\
        &G(w) = b_0 + b_1 w + b_2 w^2 + \cdots
    \end{align}
    where $a_n, b_m \in \Z_p$.
    We assumed $|a_0|, |b_0| < 1 $.
    (Recall that $|\ |$ is the $p$-adic absolute value normalized as $|p|=p^{-1}$. )
    We treat the following cases separately.
    \begin{table}[h]
        \centering
        \begin{tabular}{cl}
            Case 1 &  $|a_1| = |b_1| = 1$. \\
            Case 2 & $|a_1| =1$ and $|b_1| < 1$.\\
            Case 3 & $|a_1| < 1$ and $|b_1| < 1$.
        \end{tabular}
    \end{table}
    
    {\bf \underline{Case 1.}} Suppose $|a_1| = |b_1|=1$.
    In this case, $F(z) - a_0$ and $G(w) - b_0$ have compositional inverse 
    in $z\Z_p\lb z \rb$ and $w \Z_p \lb w \rb$ respectively.
    This implies the morphism
    \begin{align}
        \f \colon \Q_p\{pz, pw\} \longrightarrow \Q_p\{pz, pw\}\ ; z \mapsto F(z), 
        w \mapsto G(w)
    \end{align}
    is an isometric isomorphism.
    In particular, $F \times G \colon (p\Z_p)^2 \longrightarrow (p\Z_p)^2$ 
    is a bijection. This proves PIQ for the case where $Z$ is a fixed point of $F \times G$.
    When $Z = (\Phi=0)$, recall that we have $\Phi \mid \f(\Phi)$ in $\Q_p\{pz, pw\} $.
    Since $\f$ is a ring automorphism, the number of prime factors of $\f(\Phi)$
    counted with multiplicity is equal to that of $\Phi$.
    Hence we get $\f(\Phi) = \g \Phi$ for some $\g \in \Q_p\{pz, pw\}^{\times}$.
    Thus for $(a,b) \in (p\Z_p)^2$, we have
    \begin{align}
        &(F \times G)(a,b) \in Z \iff \Phi(F(a),G(b)) = 0 \iff \f(\Phi)(a,b)=0 \\
        &\iff \g(a,b)\Phi(a,b) = 0 \iff \Phi(a,b) = 0 \iff (a,b ) \in Z. 
    \end{align}
    This proves PIQ.

    {\bf \underline{Case 2.}} Suppose $|a_1| = 1$ and $|b_1| < 1$.
    In this case, $F(z) - a_0$ has compositional inverse in $z\Z_p \lb z \rb$
    and hence 
    \begin{align}
        \Q_p \{ pz \} \longrightarrow \Q_p\{ pz \}\ ; z \mapsto F(z)
    \end{align}
    is an isometric isomorphism.
    We claim that $G \colon p\Z_p \longrightarrow p\Z_p$ has a unique fixed point.
    Indeed, consider $G(w)-w-b_0 = (b_1 -1)w + b_2 w^2 + \cdots \in \Z_p \lb w \rb$. Since $|b_1 - 1| = 1$, this has a compositional inverse in 
    $w \Z_p\lb w \rb$.
    Therefore, $p\Z_p \longrightarrow p \Z_p \ ; w \mapsto G(w)-w-b_0$
    is a bijection and thus $G$ has a unique fixed point in $p\Z_p$.
    
    Let $c \in p\Z_p$ be the fixed point.
    Let $\widetilde{G}(w) = G(w+c)-c \in \Z_p \lb w \rb$.
    Then we have the commutative diagram 
    \begin{equation}\label{diagram:conj-by-translation}
    \begin{tikzcd}
    &[-3em] w \arrow[r,mapsto] & G(w)  \\[-2em]
    w \ar[d, mapsto] &[-3em] \Q_p\{ pw \} \arrow[r] \arrow[d] & \Q_p\{ pw \}  \arrow[d]  &[-3em] w \ar[d, mapsto] \\
    w +c &[-3em] \Q_p\{ pw \}  \arrow[r] & \Q_p\{ pw \}  &[-3em] w + c\\[-2em]
    & w \ar[r,mapsto] & \widetilde{G}(w) &
    \end{tikzcd}
    \end{equation}
    where the vertical arrows are isometric isomorphisms.
    We also have
    \begin{align}
        &\widetilde{G}(0) = 0\\
        &\widetilde{G}'(0) = G'(c) = b_1 + 2b_2c + \cdots, \ \text{and hence}\ 
        |\widetilde{G}'(0)| \leq p^{-1}.
    \end{align}
    Note that for arbitrary $b \in \C_p$ with $|b| \leq p^{-1}$
    (we fix a $\C_p$, completion of the algebraic closure of $\Q_p$, and will work inside it), we have
    \begin{align}
        |\widetilde{G}(b)| \leq \max\{ |\widetilde{G}'(0)|, |b| \}|b| \leq p^{-1}|b|
    \end{align}
    and hence 
    \begin{align}\label{ineq:G-is-contraction-map}
        |\widetilde{G}^{\circ n}(b)| \leq p^{-n-1}
    \end{align}
    for all $n \in \Z_{\geq 0}$.
    To prove PIQ for $F \times G$, it is enough to prove it for 
    $F \times \widetilde{G}$ acting on $(p\Z_p)^2$, and so we replace 
    $G$ with $\widetilde{G}$.
    
    {\bf Case 2-1.}
    First we treat the case where $G'(0) =: \l \neq 0$.
    Let us write $|\l| = p^{-r}$ for some $r \in \Z_{\geq 1}$.

    By \cref{thm:p-adic-unif,lem:isom-of-small-discs}, we have the following commutative diagram
    \begin{equation}
    \begin{tikzcd}
     w \arrow[r,mapsto] & G(w)  \\[-2em]
     \Q_p\{ pw \} \arrow[r] \arrow[d,hookrightarrow] & \Q_p\{ pw \}  \arrow[d,hookrightarrow]   \\
     \Q_p\{ p^n w \}  \arrow[r] \arrow[d,"\iota",swap] & \Q_p\{ p^n w \} \arrow[d,"\iota"] \\
    \Q_p\{ p^N w \} \arrow[r] & \Q_p\{ p^N w \}\\[-2em]
    w \ar[r,mapsto] & \l w &
    \end{tikzcd}
    \end{equation}
    where $n, N \in \Z_{\geq 1}$, the top vertical arrows are inclusions,
    and $\iota$ is an isometric isomorphism such that $\iota(w)(0)=0$.
    Taking completed tensor product with $\Q_p\{ pz \}$, we get
    \begin{equation}
        \begin{tikzcd}
            \Q_p\{ pz, pw \} \arrow[dd,"\chi",swap,bend right=70] \arrow[r,"\substack{z \mapsto F(z) \\ w \mapsto G(w)}"] \arrow[d,hookrightarrow] & \Q_p\{ pz, pw \} \arrow[d,hookrightarrow] \\
            \Q_p\{ pz, p^n w \} \arrow[r] \arrow[d,"\sim", sloped] \arrow[d,phantom,"\scriptsize \text{$\id \widehat{\otimes} \iota$}"',shift right=1.2em]& \Q_p\{ pz, p^n w \} \arrow[d,"\sim", sloped] \arrow[d,phantom,"\scriptsize \text{$\id \widehat{\otimes} \iota$}"',shift right=1.2em] \\
            \Q_p\{ pz, p^N w \} \arrow[r,"\a","\substack{z \mapsto F(z) \\ w \mapsto \l w}"'] & \Q_p\{ pz, p^N w \}
        \end{tikzcd}
    \end{equation}
    where $\chi$ is the composite of the vertical arrows and
    $\a$ is the $\Q_p$-algebra morphism defined 
    by sending $z$ to $F(z)$ and $w$ to $\l w$.
    The map $\chi$ induces 
    \begin{align}
        (p\Z_p)^2 \supset p\Z_p \times p^n\Z_p \xleftarrow[(\id \widehat{\otimes} \iota)^*]{\sim} p\Z_p \times p^N\Z_p.
    \end{align}
    Let 
    \begin{align}
        Z' = \left((\id \widehat{\otimes} \iota)^*\right)^{-1}(Z \cap (p\Z_p \times p^n\Z_p)) \subset p\Z_p \times p^N\Z_p.
    \end{align}
    By \cref{ineq:G-is-contraction-map}, we have
    \begin{align}
        (F \times G)^{\circ (n-1)}( (p\Z_p)^2 ) \subset p\Z_p \times p^n \Z_p. 
    \end{align}
    Therefore, it is enough to prove PIQ for the map
    $F(z) \times \l w$ acting on $p\Z_p \times p^N \Z_p$
    and $Z'$.
    When $Z'$ consists of a single fixed point, PIQ is clear
    because $F \colon p\Z_p \longrightarrow p\Z_p$ is a bijection and the only fixed point of 
    $\l w$ is $0$.
    Consider the case $Z = (\Phi = 0)$.
    In this case, set $\Psi = \chi(\Phi)$.
    By \cref{cor:prime-fac-bdd}, there is arbitrary large $N' \geq N$ such that
    for all $m \geq N'$, we have
    \begin{align}
        &\text{number of prime factors of $\Psi$ in $\Q_p\{pz, p^m w\}$ counted with multiplicity} \\
        &\leq \text{that of $\Psi$ in $\Q_p\{pz, p^{N'} w\}$}.
    \end{align}
    We replace $N$ with such large $N'$ (cf.\ \cref{lem:isom-of-small-discs}).
    Then, consider the decomposition 
    \begin{equation}
        \begin{tikzcd}
            \Q_p\{pz, p^N w\} \arrow[dr,hookrightarrow] \arrow[rr,"\a"] &[-2em] &[-2em] \Q_p\{pz, p^N w\} \\
            & \Q_p\{pz, p^{N+r} w\} \arrow[ru,"\substack{z\mapsto F(z) \\ w \mapsto \l w}",swap, sloped, "\sim"'] &
        \end{tikzcd}
    \end{equation}
    where the first map is inclusion, and the second one is an isometric isomorphism.
    By this diagram, the number of prime factors of $\a(\Psi)$
    is at most that of $\Psi$ by the choice of $N$.
    Then we have
    \begin{align}
        \a(\Psi) = \g \Psi
    \end{align}
    for some $\g \in \Q_p\{pz, p^N w\}^{\times}$
    because $\Psi \mid \a(\Psi)$.
    Therefore, for $(a,b) \in p\Z_p \times p^N \Z_p$, we have
    \begin{align}
        &(F(a),\l b) \in Z' \iff \Psi((F(a),\l b)) = 0 
        \iff \a(\Psi)(a,b)=0 \\
        &\iff \g(a,b)\Psi(a,b) = 0 
        \iff \Psi(a,b) = 0 \iff (a,b ) \in Z'
    \end{align}
    and PIQ follows from this.

    {\bf Case 2-2.}
    Next suppose $G'(0)=0$.
    Let 
    \begin{align}
        G(w) = b_d w^d + \text{higher order terms}
    \end{align}
    with $b_d \neq 0$.
    Set $K = \Q_p( b_d^{1/(d-1)} )$ where $b_d^{1/(d-1)}$ is a fixed 
    $(d-1)$-th root of $b_d$.
    By \cref{thm:p-adic-unif,lem:isom-of-small-discs}, 
    we have the following commutative diagram
    \begin{equation}
    \begin{tikzcd}
     w \arrow[r,mapsto] & G(w)  \\[-2em]
     K\{ pw \} \arrow[r] \arrow[d,hookrightarrow] & K\{ pw \}  \arrow[d,hookrightarrow]   \\
     K\{ p^n w \}  \arrow[r] \arrow[d,"\iota",swap] & K\{ p^n w \} \arrow[d,"\iota"] \\
    K\{ p^N w \} \arrow[r] & K\{ p^N w \}\\[-2em]
    w \ar[r,mapsto] & w^d &
    \end{tikzcd}
    \end{equation}
    where $n, N \in \Z_{\geq 1}$, the top vertical arrows are inclusions,
    and $\iota$ is an isometric isomorphism such that $\iota(w)(0)=0$.
    Taking completed tensor product with $K\{ pz \}$, we get
    \begin{equation}
        \begin{tikzcd}
            K\{ pz, pw \} \arrow[dd,"\chi",swap,bend right=70] \arrow[r,"\substack{z \mapsto F(z) \\ w \mapsto G(w)}"] \arrow[d,hookrightarrow] & K\{ pz, pw \} \arrow[d,hookrightarrow] \\
            K\{ pz, p^n w \} \arrow[r] \arrow[d,"\sim", sloped] \arrow[d,phantom,"\scriptsize \text{$\id \widehat{\otimes} \iota$}"',shift right=1.2em] & K\{ pz, p^n w \} \arrow[d,"\sim", sloped] \arrow[d,phantom,"\scriptsize \text{$\id \widehat{\otimes} \iota$}"',shift right=1.2em] \\
            K\{ pz, p^N w \} \arrow[r,"\a","\substack{z \mapsto F(z) \\ w \mapsto w^d}"'] & K\{ pz, p^N w \}
        \end{tikzcd}
    \end{equation}
    where $\chi$ is the composite of the vertical arrows and
    $\a$ is the $K$-algebra morphism defined 
    by sending $z$ to $F(z)$ and $w$ to $w^d$.
    Set $\Psi=\chi(\Phi)$.
    As before, we have 
    \begin{align}
        (F \times G)^{\circ (n-1)}(\bD_K(p^{-1})^2) 
        \subset \bD_K(p^{-1}) \times \bD_K(p^{-n}).
    \end{align}
    Thus, it is enough to prove PIQ for $F(z) \times w^d$ acting 
    on $\bD_K(p^{-1}) \times \bD_K(p^{-N})$ and the invariant subset $Z'$ where   
    \begin{align}
        Z' =
        \begin{cases}
            \text{a fixed point, or}\\
            \{(a,b) \in \bD_K(p^{-1}) \times \bD_K(p^{-N}) \mid \Psi(a,b)=0 \}.
        \end{cases}
    \end{align}
    When $Z'$ is a fixed point, PIQ holds because 
    $F \colon \bD_K(p^{-1}) \longrightarrow \bD_K(p^{-1})$ is bijective and
    $w^d$ has only one fixed point, $0$, in $\bD_K(p^{-N})$.
    Suppose $Z' = \{(a,b) \in \bD_K(p^{-1}) \times \bD_K(p^{-N}) \mid \Psi(a,b)=0 \}$.

    We first note that we may assume $\Psi$ is prime.
    Indeed, we claim that $\a$ acts on the set of minimal prime ideals
    containing $\Psi$ by pull-back.
    Let $\p \subset K\{pz, p^N w\}$ be a minimal prime ideal containing $\Psi$.
    Then we have $\Psi \in \a^{-1}(\p)$ and thus it is enough to show that
    $\a^{-1}(\p)$ is not a maximal ideal.
    Consider the following decomposition.
    \begin{equation}
        \begin{tikzcd}
            K\{pz, p^N w\} \arrow[dr,hookrightarrow,"i"] \arrow[rr,"\a"] &[-2em] &[-2em] K\{pz, p^N w\} \\
            & K\{pz, p^{Nd} w\} \arrow[ru,"\substack{z\mapsto F(z) \\ w \mapsto w^d}",swap, sloped,"j"'] &
        \end{tikzcd}
    \end{equation}
    The second map $j$ is finite injective and thus $j^{-1}(\p)$ is a height one 
    prime ideal.
    Then $i^{-1}(j^{-1}\p) = \a^{-1}(\p)$ is not a maximal ideal
    because $K\{pz, p^{Nd} w\}$ is a Jacobson ring and Nullstellensatz holds.
    Finally, we replace $\a$ with some iterates so that
    every minimal prime of $\Psi$ is invariant or mapped to an invariant one.
    Hence, it is enough to prove PIQ for invariant minimal primes, and this means 
    we may assume $\Psi$ is a prime.
    
    Now we consider the case $w \mid \Psi$.
    In this case we can write $\Psi = \g w$ for some 
    $\g \in K\{pz, p^Nw\}^{\times}$. 
    Thus for $(a,b) \in \bD_K(p^{-1}) \times \bD_K(p^{-N})$, 
    \begin{align}
        &(F(a), b^d) \in Z' \iff \Psi(F(a), b^d) = 0 \iff \g(F(a),b^d)b^d=0\\
        &\iff b=0 \iff (a,b) \in Z'
    \end{align}
    and we are done.

    Suppose $w \not\mid \Psi$.
    Let us write 
    \begin{align}
        \Psi(z,w) &= \Psi_0(z) + w\Psi_1(z,w) \\
        &\qquad \text{for some}\  \Psi_0 \in K\{pz\} \smallsetminus \{0\}
        \ \text{and}\ \Psi_1 \in K\{pz, p^Nw\},\\
        \a(\Psi) &= \g \Psi  \ \text{for some}\ \g \in  K\{pz, p^Nw\}.
    \end{align}
    Let us also write $\g = \g_0 + w \g_1$
    with $\g_0 \in K\{pz\}$ and $\g_1 \in K\{pz, p^Nw\}$.
    Then we have
    \begin{align}
        \a(\Psi)(z,w) &= \Psi_0(F(z)) + w^d\Psi(F(z),w^d) \\
        &= \left(\g_0(z) + w \g_1(z,w)\right) 
        \left(\Psi_0(z) + w\Psi_1(z,w) \right).
    \end{align}
    Thus we have 
    \begin{align}
        \Psi_0(F(z)) = \g_0(z) \Psi_0(z).
    \end{align}
    Since $K\{pz\} \longrightarrow K\{pz\} ; z \mapsto F(z)$
    is an isometric isomorphism, the number of prime factors of $\Psi_0(F(z))$
    is equal to that of $\Psi_0(z)$.
    Thus $\g_0 \in K\{pz\}^{\times}$.
    If we take $N' \geq N$ so that
    $\|\g_0 \|_{K\{pz\}} > \| w\g_1\|_{K\{pz, p^{N'}w\}}$,
    then $\g \in K\{pz, p^{N'}w\}^{\times}$.
    Therefore, for $(a,b) \in \bD_K(p^{-1}) \times \bD_K(p^{-N})$,
    \begin{align}
        &(F^{\circ n}(a), b^{d^n}) \in Z' \iff \Psi(F^{\circ n}(a), b^{d^n}) =0\\
        &\iff 
        \g(F^{\circ (n-1)}(a), b^{d^{n-1}}) \Psi(F^{\circ (n-1)}(a), b^{d^{n-1}})=0\\
        &\iff \Psi(F^{\circ (n-1)}(a), b^{d^{n-1}})=0\\
        &\iff (F^{\circ (n-1)}(a), b^{d^{n-1}}) \in Z'
    \end{align}
    if $|b^{d^{n-1}}| \leq p^{-N'}$, which is true when
    $p^{-Nd^{n-1}} \leq p^{-N'}$.
    This proves PIQ for this case.

    {\bf \underline{Case 3.}}
    Finally, we treat the case where $|a_1|,|b_1|<1$.
    As at the beginning of Case 2,
    we may assume $F(0)=G(0)=0$ by taking conjugate by translation.
    Moreover, we have
    \begin{align}
        &|F'(0)| \leq p^{-1},\ |G'(0)| \leq p^{-1}\\
        &(F \times G)^{\circ n} \left(\bD_{\C_p}(p^{-1})^2\right) \subset \bD_{\C_p}(p^{-n-1})^2.
        \label{eq:FG-is-contraction}
    \end{align}

    {\bf Case 3-1.}
    Let us first treat the case $F'(0)=: \l \neq 0$.
    By \cref{thm:p-adic-unif,lem:isom-of-small-discs}, there is a finite extension
    $K$ of $\Q_p$ and the following commutative diagram.
    \begin{equation}
        \begin{tikzcd}
            K\{ pz, pw \} \arrow[dd,bend right=70,"\chi",swap] \arrow[r,"\substack{z \mapsto F(z) \\ w \mapsto G(w)}"] \arrow[d,hookrightarrow] & K\{ pz, pw \} \arrow[d,hookrightarrow]\\
            K\{ p^nz, p^nw \} \arrow[d,"\sim",sloped] \arrow[d,phantom, "\text{\scriptsize$\iota$}",shift right={.5em}] \arrow[r] & K\{ p^nz, p^nw \} \arrow[d,"\sim",sloped] \arrow[d,phantom, "\text{\scriptsize$\iota$}",shift right={.5em}]\\
            K\{ p^Nz, p^Nw \} \arrow[r,"\a","\substack{z \mapsto \l z \\ w \mapsto g(w)}"'] & K\{ p^Nz, p^Nw \}
        \end{tikzcd}
    \end{equation}
    Here, $n,N \in \Z_{\geq 1}$,
    the top vertical arrows are inclusion, $\iota$ is an isometric isomorphism, 
    and $\a$ is the $K$-algebra morphism defined by sending 
    $z$ to $\l z$ and $w$ to $g(w)$, where $g(w)= \mu w$ for some $\mu \in p\Z_p \smallsetminus \{0\}$
    or $g(w)=w^d$ for some $d\in \Z_{\geq 2}$.
    Let $\Psi = \chi(\Phi)$.
    By \cref{eq:FG-is-contraction}, it is enough to prove PIQ for the map
    $\l z \times g$ acting on $\bD_K(p^{-N})^2$ and the invariant subset 
    \begin{align}
        Z' =
        \begin{cases}
            \text{a fixed point, or}\\
            \{(a,b) \in \bD_K(p^{-N})^2 \mid \Psi(a,b)=0 \}.
        \end{cases}
    \end{align}
    When $Z'$ is a fixed point, PIQ is clear because the only fixed point is $(0,0)$
    and its only preimage is itself.
    Suppose $Z'=(\Psi=0)$.
    As in Case 2-2, by replacing $\a$ with iterates, we may assume $\Psi$ is prime.
    
    Consider the case $w \mid \Psi$.
    In this case we can write $\Psi = \g w$ for some 
    $\g \in K\{p^Nz, p^Nw\}^{\times}$. 
    Thus for $(a,b) \in \bD_K(p^{-N})^2$, 
    \begin{align}
        &(\l a, g(b)) \in Z' \iff \Psi(\l a, g(b)) = 0 \iff \g(\l a,g(b)) g(b)=0\\
        &\iff b=0 \iff (a,b) \in Z'
    \end{align}
    and we are done.

    Suppose $w \not\mid \Psi$.
    Let us write 
    \begin{align}
        &\Psi(z,w) = \Psi_0(z) + w\Psi_1(z,w) \\
        & \qquad \text{for some}\  \Psi_0 \in K\{p^Nz\} \smallsetminus \{0\}
        \ \text{and}\ \Psi_1 \in K\{p^Nz, p^Nw\},\\
        &\a(\Psi) = \g \Psi  \ \text{for some}\ \g \in  K\{p^Nz, p^Nw\}.
    \end{align}
    Let us also write $\g = \g_0 + w \g_1$
    for some $\g_0 \in K\{p^Nz\}$ and $\g_1 \in K\{p^Nz, p^Nw\}$.
    Then we have
    \begin{align}
        \alpha(\Psi)(z,w) &=
        \Psi_0(\l z) + g(w) \Psi_1(\l z, g(w)) \\
        &= \left( \g_0(z) + w \g_1(z,w)\right) \left( \Psi_0(z) + w \Psi_1(z,w)\right)
    \end{align}
    and thus
    \begin{align}
        \Psi_0(\l z) = \g_0(z) \Psi_0(z).
    \end{align}
    Since $\ord_{K\{p^Nz\}} \Psi_0(\l z) \leq \ord_{K\{p^Nz\}} \Psi_0(z) $ by $|\lambda|<1$,
    we get $\ord_{K\{p^Nz\}} \g_0(z) = 0$, i.e.\ $\g_0 \in K\{p^Nz\}^{\times}$.
    Then there is $N' \geq N$ such that $\g \in K\{p^Nz, p^{N'}w\}^{\times}$.
    Now for $(a,b) \in \bD_K(p^{-N})^2$, we have
    \begin{align}
        &(\l^n a, g^{\circ n}(b)) \in Z' \iff \Psi(\l^n a, g^{\circ n}(b))=0\\
        &\iff \g(\l^{n-1} a, g^{\circ (n-1)}(b)) \Psi(\l^{n-1} a, g^{\circ (n-1)}(b))=0 \\
        &\iff \Psi(\l^{n-1} a, g^{\circ (n-1)}(b)) =0
    \end{align}
    if $|g^{\circ (n-1)}(b)| \leq p^{-N'}$, which is true when
    $p^{-N-(n-1)} \leq p^{-N'}$. This proves PIQ for this case.

    {\bf Case 3-2.}
    Finally we treat the case $F'(0)=G'(0)=0$.
    Let
    \begin{align}
        &F(z) = a_d z^d + \text{higher order terms}\\
        &G(w) = b_e w^e + \text{higher order terms}
    \end{align}
    where $d,e \in \Z_{\geq 2}$ and $a_d, b_e \neq 0$.

    By \cref{thm:p-adic-unif,lem:isom-of-small-discs}, there is a finite extension
    $K$ of $\Q_p$ and the following commutative diagram.
    \begin{equation}
        \begin{tikzcd}
            K\{ pz, pw \} \arrow[dd,bend right=70,"\chi",swap] \arrow[r,"\substack{z \mapsto F(z) \\ w \mapsto G(w)}"] \arrow[d,hookrightarrow] & K\{ pz, pw \} \arrow[d,hookrightarrow]\\
            K\{ p^nz, p^nw \} \arrow[d,"\sim",sloped] \arrow[d,phantom, "\text{\scriptsize$\iota$}",shift right={.5em}] \arrow[r] & K\{ p^nz, p^nw \} \arrow[d,"\sim",sloped] \arrow[d,phantom, "\text{\scriptsize$\iota$}",shift right={.5em}]\\
            K\{ p^Nz, p^Nw \} \arrow[r,"\a","\substack{z \mapsto z^d \\ w \mapsto w^e}"'] & K\{ p^Nz, p^Nw \}
        \end{tikzcd}
    \end{equation}
    Here, $n,N \in \Z_{\geq 1}$,
    the top vertical arrows are inclusion, $\iota$ is an isometric isomorphism, 
    and $\a$ is the $K$-algebra morphism defined by sending 
    $z$ to $z^e$ and $w$ to $w^e$.
    Let $\Psi = \chi(\Phi)$.
    By \cref{eq:FG-is-contraction}, it is enough to prove PIQ for the map
    $z^d \times w^e$ acting on $\bD_K(p^{-N})^2$ and the invariant subset 
    \begin{align}
        Z' =
        \begin{cases}
            \text{a fixed point, or}\\
            \{(a,b) \in \bD_K(p^{-N})^2 \mid \Psi(a,b)=0 \}.
        \end{cases}
    \end{align}
    When $Z'$ is a fixed point, PIQ is clear because the only fixed point is $(0,0)$
    and its only preimage is itself.
    Suppose $Z'=(\Psi=0)$.

    Let us define a sequence of finite extensions of $K$ as follows.
    Let $K_0$ be a finite extension of $K$ such that
    the prime factorization of $\Psi \in K_0\{p^Nz, p^Nw\}$ stabilizes along 
    ground field extensions (such $K_0$ exists by \cref{cor:prime-fac-fld-ext}).
    Next, let $K_1$ be a finite extension of 
    $K_0$ such that the prime factorization of $\Psi \in K_1\{p^{Nd}z, p^{Ne}w\}$ stabilizes.
    Having defined $K_0, K_1, \dots, K_m$, the next finite extension $K_{m+1}/K_m$ is defined
    such that the prime factorization of $\Psi \in K_{m+1}\{p^{Nd^{m+1}}z, p^{Ne^{m+1}}w\}$ stabilizes.
    Inductively we get $K_0 \subset K_1 \subset \cdots$.
    We set $R_m = K_m \{p^{Nd^m}z , p^{Ne^m}w \}$. Then we have
    \begin{align}\label{seq:Rm}
        R_0 \subset R_1 \subset R_2 \subset \cdots.
    \end{align}
    By \cref{lem:prim-fact-ev-stab}, prime factorization of $\Psi$ in $R_m$
    eventually stabilizes in the following sense:
    There is $m_0 \geq 0$ such that 
    \begin{align}
        \Psi = \g P_1 \cdots P_l \ \text{in $R_{m_0}$}
    \end{align}
    where $\g \in R_{m_0}$ is eventually unit with respect to \cref{seq:Rm},
    $P_i \in R_{m_0}$ is prime, and for any $m \geq m_0$,
    $P_i = v P_i'$ in $R_m$ for some prime $P_i' \in R_m$ and eventual unit $v \in R_m$.

    To prove PIQ, we may replace $\a \colon K\{p^Nz, p^Nw\} \longrightarrow K\{p^Nz, p^Nw\}$ with
    \begin{align}
        K_{m_0} \{p^{Nd^{m_0}}z , p^{Ne^{m_0}}w \} \longrightarrow
        K_{m_0} \{p^{Nd^{m_0}}z , p^{Ne^{m_0}}w \}\ ; z \mapsto z^d,
         w \mapsto w^e.
    \end{align}
    Then, as in Case 2-2, it is enough to prove PIQ for an iterate $\a^t$ of $\a$,
    and each prime factor $Q$ of $\Psi$ in $R_{m_0}$ that is invariant under $\a^t$
    i.e.\ $Q \mid \a^t(Q)$.
    Consider $L = K_{m_0 + t}(\zeta_{d^t},\zeta_{e^t})$, 
    where $\zeta_{l}$ stands for an $l$-th primitive root of $1$ in $\C_p$.
    (We add $\zeta_{d^t}, \zeta_{e^t}$ for the later use.)
    Then, form the following commutative diagram
    \begin{equation}
        \begin{tikzcd}
            L\{p^{Nd^{m_0}}z, p^{Ne^{m_0}}w  \} \arrow[rr,"\a^t"] \arrow[rd,hookrightarrow]
            &[-1.5em] &[-1.5em] L\{p^{Nd^{m_0}}z, p^{Ne^{m_0}}w  \}\\
            & L\{p^{Nd^{m_0+t}}z, p^{Ne^{m_0+t}}w  \}  
            \arrow[ru,"\substack{z \mapsto z^{d^t} \\ w \mapsto w^{e^t}}",sloped,swap] &
        \end{tikzcd}
    \end{equation}
    where the first slant arrow is the inclusion.
    By the choice of $K_m$ and $L$, the element $Q$ is still prime in $L\{p^{Nd^{m_0}}z, p^{Ne^{m_0}}w  \}$.
    Moreover, either 
    \begin{itemize}
        \item $Q$ is eventually unit with respect to $(L\{p^{Nd^m}z, p^{Ne^m}w\})_m$, or
        \item $Q = vQ'$ in $L\{p^{Nd^{m_0+t}}z, p^{Ne^{m_0+t}}w  \} $, where
        $Q'$ is a prime element of \\
        $L\{p^{Nd^{m_0+t}}z, p^{Ne^{m_0+t}}w  \}$ and 
        $v$ is eventually unit with respect to \\
        $(L\{p^{Nd^m}z, p^{Ne^m}w\})_m$.
    \end{itemize}
    Here, we claim that the first case does not happen.
    Indeed, if $Q$ is eventually a unit, then it is a unit in
    $L\{p^{Nd^{m_0+kt}}z, p^{Ne^{m_0+kt}}w\}$ for some $k \geq 1$.
    We have $Q \mid \a^t(Q)$ and this implies $Q \mid \a^{tk}(Q)$ in 
    $L\{ p^{Nd^{m_0}}z, p^{Ne^{m_0}}w\}$.
    Since $\a^{tk}$ factors through $L\{p^{Nd^{m_0+kt}}z, p^{Ne^{m_0+kt}}w\}$,
    $\a^{tk}(Q)$ is unit and hence so is $Q$. This contradicts to $Q$ being prime.
    
    Let us set $r = p^{-Nd^{m_0}}$ and $s = p^{-Ne^{m_0}}$.
    Then, it is enough to prove PIQ for
    \begin{align}
        \a^t \colon L\{r^{-1}z, s^{-1}w  \} \longrightarrow L\{r^{-1}z, s^{-1}w  \}
    \end{align}
    and $Q$.
    
    We reset the notation: we write $K$ in place of $L$,
    $\a$ in place of $\a^t$, $\Psi$ in place of $Q$, $\Psi'$ in place of $Q'$,
    and $\g$ in place of $v$.
    We remark that by this replacement, the original degrees $d,e$ are replaced with
    their power $d^t, e^t$.

    The situation is summarized as follows.
    \begin{itemize}
        \item $K$ is a finite extension of $\Q_p$ such that $\zeta_d, \zeta_e \in K$.
        \item $r,s \in |K^{\times}|$, $r,s \leq p^{-1}$.
        \item $\Psi \in K\{r^{-1}z, s^{-1}w  \}$ is a prime element that is not eventually unit with respect to $(K\{r^{-d^m}z,s^{-e^m}w\})_m$.
        \item $\Psi = \g \Psi'$ in $K\{r^{-d}z, s^{-e}w \} $, where $\Psi'$ is a prime in $K\{r^{-d}z, s^{-e}w \}$ and $\g$ is eventually unit with respect to $(K\{r^{-d^m}z,s^{-e^m}w\})_m$.
    \end{itemize}
    We have the following commutative diagram.
    \begin{equation}
        \begin{tikzcd}
            \Psi \arrow[rd,mapsto] \arrow[r,phantom,"\in"] 
            &[-1.5em] K\{r^{-1}z, s^{-1}w  \} \arrow[rr,"\a"] \arrow[rd,hookrightarrow]
            &[-2em] &[-2em] K\{r^{-1}z, s^{-1}w   \}\\[1em]
            & |[xshift=3em]| \g \Psi' \arrow[r,phantom,"\in"] & K\{r^{-d}z, s^{-e}w \}  
            \arrow[ru,"\substack{z \mapsto z^{d} \\ w \mapsto w^{e}}",sloped,swap,"\widetilde{\a}"'] &
        \end{tikzcd}
    \end{equation}
    Here, we call the second slant arrow $\widetilde{\a}$.

    If $z \mid \Psi$, then we have $\Psi = u z$ for some $u \in K\{r^{-1}z, s^{-1}w\}$
    because $\Psi$ is prime.
    Then for $(a,b) \in \bD_K(r) \times \bD_K(s)$, we have
    \begin{align}
        \Psi(a^d, b^e)=0 \iff u(a^d, b^e)a^d=0 \iff a=0 \iff \Psi(a,b)=0
    \end{align}
    and PIQ follows.
    By the same reason, we are done if $w \mid \Psi$.
    
    Now we assume $z \not\mid \Psi$ and $w \not\mid \Psi$.
    We claim that $d=e$.
    First note that have $\Psi(0,0)=0$ because $\Psi$ is not eventually unit.
    Recall that we have $\Psi \mid \a(\Psi) = \widetilde{\a}(\g)\widetilde{\a}(\Psi')$.
    Since $\g$ is eventually a unit, it is easy to see that
    $\widetilde{\a}(\g)$ is also eventually unit with respect to
    $\left(K\{r^{-d^m}z,s^{-e^m}w\}\right)_m$.
    Thus $\Psi$ cannot divide $\widetilde{\a}(\g)$ 
    and hence $\Psi \mid \widetilde{\a}(\Psi')$.

    Let us consider the action of 
    $G:=\langle \zeta_d \rangle \times \langle \zeta_e \rangle $
    on $K\{ r^{-1}z, s^{-1}w \}$ defined by
    \begin{align}
        (\zeta_d^i, \zeta_e^j)\cdot f(z,w) = f(\zeta_d^i z, \zeta_e^j w).
    \end{align}
    We have $K\{ r^{-1}z, s^{-1}w \}^G = \widetilde{\a}\left(K\{r^{-d}z, s^{-e}w \}  \right)$.
    Thus, $G$ acts transitively on the set of prime ideals of 
    $K\{ r^{-1}z, s^{-1}w \}$ that lie over $\widetilde{\a}(\Psi')$.
    Thus all the prime factors of $\widetilde{\a}(\Psi')$ are of the form
    $\Psi(\zeta_d^i z, \zeta_e^j w)$.
    Thus we have
    \begin{align}\label{eq:fac-alphaPsi}
        \a(\Psi) = \Psi(z^d,w^e) 
        = \widetilde{\a}(\g) u \prod_{i=1}^k \Psi(\zeta_i z, \eta_i w)
    \end{align}
    for some $u \in K\{r^{-1}z, s^{-1}w\}^{\times }$,
    $k \in \Z_{\geq 1}$, and $(\zeta_i, \eta_i) \in G$.
    Let us write
    \begin{align}
        \Psi(z,w) &= \Psi_0(z) + \Psi_1(w) + zw \Psi_2(z,w)\\
        \widetilde{\a}(\g) u &= c + \d_0(z) + \d_1(w) + zw\d_2(z,w)
    \end{align}
    where $\Psi_0, \d_0 \in zK\{r^{-1}z\}$, $\Psi_1, \d_1 \in wK\{s^{-1}w\}$,
    $\Psi_2, \d_2 \in K\{r^{-1}z, s^{-1}w\}$, and $c \in K$.
    Note that $c \neq 0$ since $\widetilde{\a}(\g)u$ is eventually unit.
    By substituting into \cref{eq:fac-alphaPsi}, we have
    \begin{align}
        &\hphantom{=} \Psi_0(z^d) + \Psi_1(w^e) + z^dw^e \Psi_2(z^d,w^e) \\
        &= \left(c + \d_0(z) + \d_1(w) + zw\d_2(z,w) \right)\\
        & \quad \times \prod_{i=1}^k  \left(\Psi_0(\zeta_i z) + \Psi_1(\eta_i w) + \zeta_i \eta_i zw \Psi_2(\zeta_i z,\eta_i w) \right).
    \end{align}
    Therefore we get
    \begin{align}
        \Psi_0(z^d) &= \left(c + \d_0(z) \right)\prod_{i=1}^k  \Psi_0(\zeta_i z)\\
        \Psi_1(w^e) &= \left(c +  \d_1(w) \right)\prod_{i=1}^k  \Psi_1(\eta_i w).
    \end{align}
    We get $d = k = e$ by comparing the lowest degree terms.

    Next, we will detect the shape of the zero set of $\Psi$.
    In the following claim, we use the fact $p \not\mid d$.
    This can be seen as follows.
    By our assumption, $p$ does not divide any ramification indices of 
    original maps $f,g \colon \P^1 \longrightarrow \P^1$.
    In the course of reduction in Step 1 and Step 2 so far,
    we replaced $f,g$ with iterates and $d$ with its power.
    Therefore, $d$ is a ramification index of iterated $f$ (and $g$),
    and by the chain rule, we see that $d$ is a product of
    ramification indices of $f$. Thus, $p$ does not divide $d$.
    
    \begin{claim}\label{claim:zero-locus-Psi}
        The set $\left\{ (a,b) \in \bD_{\C_p}(r) \times \bD_{\C_p}(s) \ \middle|\ \Psi(a,b)=0 \right\}$ is contained in
        \begin{align}
            &\hphantom{\cup} \left(\bD_{\C_p}(r) \times \{0\}\right)
            \cup \left(\{0\} \times \bD_{\C_p}(s)\right)\\
            &\cup \bigcup_{k,l \geq 1} \left\{ (a,b) \in \bD_{\C_p}(r) \times \bD_{\C_p}(s) \ \middle|\ a^k-b^l=0 \right\}
        \end{align}
    \end{claim}
    \begin{proof}[Proof of \cref{claim:zero-locus-Psi}]
        Let us write 
        \begin{align}
            \Psi = \sum_{i,j \geq 0} c_{ij}z^iw^j.
        \end{align}
        Let $a,b \in \C_p$ be such that $|a| \leq r$, $|b| \leq s$, and
        $\Psi(a,b)=0$.
        It is enough to show that $a^k-b^l=0$ for some $k,l \geq 1$
        under the assumption $ab \neq 0$.
        Let
        \begin{align}
            M &= \max\{|a^ib^j| \mid c_{ij} \neq 0  \} \text{ and}\\
            S &= \{ (i,j) \in \Z_{\geq 0}^2 \mid |a^ib^j|=M, c_{ij} \neq 0 \}.
        \end{align}
        Since $|a|,|b| < 1$, $M$ exists and $S$ is a finite set.
        Let us take $(i_0,j_0) \in S$ such that the first coordinate is the smallest 
        among elements of $S$.

        Since $\Psi \mid \a(\Psi) = \Psi(z^d,w^d)$,
        we have $\Psi(a^{d^n}, b^{d^n})=0$ for all $n \geq 1$.
        Thus 
        \begin{align}
            0 &= \frac{\Psi(a^{d^n}, b^{d^n})}{(a^{i_0}b^{j_0})^{d^n} } 
            = \frac{1}{(a^{i_0}b^{j_0})^{d^n} }\sum_{i,j \geq 0} c_{ij} (a^ib^j)^{d^n}\\
            & =  \sum_{(i,j) \in S} c_{ij} \frac{(a^ib^j)^{d^n}}{(a^{i_0}b^{j_0})^{d^n}}
            +  \sum_{(i,j) \notin S} c_{ij}\frac{(a^ib^j)^{d^n}}{(a^{i_0}b^{j_0})^{d^n}}
        \end{align}
        for all $n \geq 1$.
        If $(i,j) \notin S$, then we have
        \begin{align}
            &\left| \frac{c_{ij}(a^ib^j)^{d^n}}{(a^{i_0}b^{j_0})^{d^n}} \right|
            = \left| \frac{c_{ij}a^ib^j}{a^{i_0}b^{j_0}} \right|
            \left| \frac{(a^ib^j)^{d^n-1}}{(a^{i_0}b^{j_0})^{d^n-1}} \right|\\
            &\leq \frac{\| \Psi \|_{K\{r^{-1}z,s^{-1}w\}}}{|a^{i_0}b^{j_0}|}
            \left| \frac{a^ib^j}{a^{i_0}b^{j_0}} \right|^{d^n-1}
            \leq \frac{\| \Psi \|_{K\{r^{-1}z,s^{-1}w\}}}{|a^{i_0}b^{j_0}|}
            \left(\frac{1}{p^{1/\d}}\right)^{d^n-1}
        \end{align}
        where $\d \in \Z_{\geq 1}$ is independent of $i,j$.
        By this, we get
        \begin{align}
            \left|\sum_{(i,j) \notin S} c_{ij}\frac{(a^ib^j)^{d^n}}{(a^{i_0}b^{j_0})^{d^n}} \right| \leq \frac{\| \Psi \|_{K\{r^{-1}z,s^{-1}w\}}}{|a^{i_0}b^{j_0}|}
            \left(\frac{1}{p^{1/\d}}\right)^{d^n-1} \xrightarrow{n \to \infty} 0.
        \end{align}
        Thus we have
        \begin{align}
            \lim_{n \to \infty} \sum_{(i,j) \in S} c_{ij} \frac{(a^ib^j)^{d^n}}{(a^{i_0}b^{j_0})^{d^n}} =0.
        \end{align}
        Now, since $|a^{i-i_0}b^{j-j_0}|=1$ for $(i,j) \in S$,
        we have
        \begin{align}
            j - j_0 = -\frac{\log|a|}{\log|b|}(i-i_0).
        \end{align}
        Let us write $-\log|a|/\log|b| = m/q$ by coprime integers
        $m,q$ with $m \leq -1$ and $q \geq 1$.
        Then fix a $q$-th root $b'$ of $b$ in $\C_p$.
        Then we have
        \begin{align}
            \frac{a^ib^j}{a^{i_0}b^{j_0}} = a^{i-i_0} b^{m(i-i_0)/q} = 
            a^{i-i_0} {b'}^{m(i-i_0)} = (a{b'}^m)^{i-i_0}.
        \end{align}
        Set $\xi = a{b'}^m$. Then we have $|\xi|=1$ and 
        \begin{align}
            \lim_{n \to \infty} \sum_{(i,j) \in S} c_{ij}\xi^{i-i_0} = 0.
        \end{align}
        By \cref{lem:seq-conv-to-roots} below, $\xi$ is a root of $1$.
        This implies $a^qb^m$ is a root of $1$, and we are done.
        
    \end{proof}

    The above claim implies the following.

    \begin{claim}\label{claim:Psidivideszkwl}
        There are $k,l \in \Z_{\geq 1}$ such that
        \begin{align}
            \Psi \mid z^k - w^l
        \end{align}
        in $K\{r^{-1}z, s^{-1}w \}$.
    \end{claim}
    \begin{proof}[Proof of \cref{claim:Psidivideszkwl}]
        Let $\Psi = u\Psi_1^{e_1} \cdots \Psi_n^{e_n}$ be the prime factorization in 
        $\C_p\{r^{-1}z, s^{-1}w\}$ with $u$ being a unit and $\Psi_i$ being primes.
        Suppose $\Psi_i$ does not divide any of $z, w$, and $z^k - w^l$ $(k,l \in \Z_{\geq 1})$
        in $\C_p\{r^{-1}z, s^{-1}w\}$.
        Let $f$ be any of $z, w$, and $z^k - w^l$.
        Then the Krull dimension of $\C_p\{r^{-1}z, s^{-1}w\}/(\Psi_i,f)$ is zero
        and hence, it is an Artinian ring.
        Thus, it has only finitely many maximal ideals, and thus
        \begin{align}
            \left\{ (a,b) \in \bD_{\C_p}(r) \times \bD_{\C_p}(s) \ \middle|\ \Psi_i(a,b)=f(a,b)=0 \right\}
        \end{align}
        is a finite set.
        By \cref{claim:zero-locus-Psi}, this implies 
        \begin{align}\label{eq:zero_set_of_Psii}
            \left\{ (a,b) \in \bD_{\C_p}(r) \times \bD_{\C_p}(s) \ \middle|\ \Psi_i(a,b)=0 \right\}
        \end{align}
        is countable.
        This is a contradiction because this set is uncountable.
        Indeed, first, take an isometric isomorphism 
        $\s_1 \colon \C_p\{r^{-1}z, s^{-1}w \} \xrightarrow{\sim} \C_p\{z,w\}$
        by scaling $z$ and $w$ (recall $r,s \in |\C_p^{\times}|$).
        Then, take isometric isomorphism 
        $\s_2 \colon \C_p\{z,w\} \xlongrightarrow{\sim} \C_p\{z,w\} $
        so that $\s_2(\s_1(\Psi_i)) = u W$, where
        $u \in \C_p\{z,w\}^{\times}$ and $W \in \C_p\{z\}[w]$ is a Weierstrass
        polynomial. Note that the degree of $W$ in $w$ is at least one since
        $\Psi_i$ is not a unit. Then we get
        \begin{align}
            \C_p\{r^{-1}z, s^{-1}w \}/(\Psi_i) \simeq \C_p\{z\}[w]/(W) 
        \end{align}
        and this is finite as a $\C_p\{z\}$-module.
        (cf.\ for example \cite[section 5.2]{BGR_NA_analysis} or
        \cite[section 2.3 and 2.4]{Menares_Tate_alg}.)
        Hence \cref{eq:zero_set_of_Psii} has a surjection onto $\bD_{\C_p}(1)$,
        which is uncountable.
        
        Thus we proved every $\Psi_i$ divides one of $z, w$, and $z^k-w^l$ $(k,l \in \Z_{\geq 1})$
        in $\C_p\{r^{-1}z, s^{-1}w\}$.
        Thus there is a finite product $\widetilde{\Psi}$ of some of
        $z, w$, and $z^k-w^l$ for $k,l \in \Z_{\geq 1}$ such that
        $\Psi \mid \widetilde{\Psi}$ in $\C_p\{r^{-1}z, s^{-1}w\}$.
        Since $\Psi, \widetilde{\Psi} \in K\{r^{-1}z, s^{-1}w\}$,
        by \cref{lem:div_and_fld_ext} below, $\Psi$ divides $\widetilde{\Psi}$ in $K\{r^{-1}z, s^{-1}w\}$.
        Finally, since $\Psi$ is prime and we assumed $z \not\mid \Psi$ and $w \not\mid \Psi$,
        there is some $k,l \in \Z_{\geq 1}$ such that $\Psi \mid z^k-w^l$.
    \end{proof}

    Now note that $z^k-w^l$ is invariant under the map $(z^d,w^d)$.
    Thus, replacing $\a$ with iterates, we may assume 
    every minimal prime ideal of $K\{r^{-1}z, s^{-1}w\}$ containing $z^k-w^l$ is either fixed or 
    mapped to a fixed one by the pullback by $\a$ (cf.\ Case 2-2).
    Let $z^k -w^l = u P_1^{e_1} \cdots P_m^{e_m}$ be the prime factorization in 
    $K\{r^{-1}z, s^{-1}w\}$ where $u$ is a unit, $P_i$ are distinct primes
    with $P_1 = \Psi$, and $e_i \in \Z_{\geq 1}$.
    Now suppose $(a,b) \in \bD_K(r) \times \bD_K(s)$ satisfies $\Psi(a^{d^n},b^{d^n})=0$
    for some $n \in \Z_{\geq 0}$.
    We want to show that there is $n_0 \geq 0$ independent of $(a,b)$ such that
    $\Psi(a^{d^{n_0}}, b^{d^{n_0}})=0$.
    First note that we have $(a^{d^n})^k - (b^{d^n})^l=0$.
    If $ab=0$, we have $a = b= 0$, and we are done.
    Suppose $ab \neq 0$.
    Then 
    \begin{align}
        \left( \frac{a^k}{b^l}\right)^{d^n} = 1.
    \end{align}
    Since $a^k/b^l \in K$ and $K$ is a finite extension of $\Q_p$, there is 
    $n_1 \geq 0$ depending only on $K$ and $d$ such that
    \begin{align}
        \left( \frac{a^k}{b^l}\right)^{d^{n_1}} = 1
    \end{align}
    or equivalently
    \begin{align}
        (a^{d^{n_1}})^k - (b^{d^{n_1}})^l = 0.
    \end{align}
    (cf.\ \cite[II (7.12),(7.13)]{Neu99}, \cite[Proposition 3.6]{BMS23}.)
    Thus there is $P_i$ whose zero locus is invariant under $(z^d,w^d)$ and such that
    $P_i(a^{d^{n_1+1}}, b^{d^{n_1+1}})=0$.
    If $i=1$, then we are done.
    Suppose $i \geq 2$.
    Then by the $(z^d,w^d)$-invariance of $(P_i=0)$ and $(\Psi=0)$ as well as our 
    assumption, we have
    \begin{align}
        &P_i(a^{d^{n_1+1+n'}}, b^{d^{n_1+1+n'}}) = 0\\
        &\Psi(a^{d^{n_1+1+n'}}, b^{d^{n_1+1+n'}}) = 0
    \end{align}
    for all large $n'$.
    Since $K\{r^{-1}z, s^{-1}w\}/(P_i, \Psi)$ is Artinian, 
    \begin{align}
        \left\{ (a',b') \in \bD_{K}(r) \times \bD_{K}(s) \ \middle|\ P_i(a',b')=\Psi(a',b')=0 \right\}
    \end{align}
    is a finite set.
    Thus there are $n'' > n' \geq 0$ such that
    \begin{align}
        (a^{d^{n_1+1+n'}}, b^{d^{n_1+1+n'}}) = (a^{d^{n_1+1+n''}}, b^{d^{n_1+1+n''}}). 
    \end{align}
    Thus we have $|a| = |b| = 1$ and this contradicts to 
    $|a|\leq r <1$.
    This completes the proof of Case 3-2, and also completes the proof of 
    \cref{prop:PIQ-P1P1-Zp}.
\end{proof}

\begin{lemma}\label{lem:seq-conv-to-roots}
    Let $p$ be a prime and $P(X) \in \C_p[X]$ be a non-zero polynomial.
    Let $d \in \Z_{\geq 1}$ be such that $p \not\mid d$.
    If a $\xi \in \C_p$ with $|\xi|=1$ satisfies 
    \begin{align}\label{eq:limP(xidn)}
        \lim_{n \to \infty} P(\xi^{d^n})=0,
    \end{align}
    then $\xi$ is a root of $1$.
\end{lemma}

\begin{proof}
    Let $\zeta_1, \dots, \zeta_r \in \C_p$ be all the distinct roots of $P$
    to which the sequence $\{\xi^{d^n}\}_n$ accumulates.
    We may replace $P$ with $(X - \zeta_1)\cdots (X - \zeta_r)$.
    Note that for each $\zeta_i$, there is $\zeta_j$ such that
    $\zeta_i^d = \zeta_j$.
    Indeed, take a subsequence $\{n_k\}_k$ such that
    $\xi^{d^{n_k}} \to \zeta_i$.
    Then $\xi^{d^{n_k+1}} \to \zeta_i^d$.
    By \cref{eq:limP(xidn)}, $\zeta_i^d$ must be equal to some $\zeta_j$.

    Now take $m \geq 1$ so that $\zeta_i^{d^m}=\zeta_i$
    or $\zeta_i \neq \zeta_i^{d^m} = \zeta_i^{d^{2m}}$ for all $i$.
    Then take small mutually disjoint open balls $B_i$ 
    around $\zeta_i$. Take small open balls $\zeta_i \in B_i' \subset B_i$
    so that the power map $z \mapsto z^{d^m}$ maps $B_i'$ into
    $B_j$ if $\zeta_i^{d^m}=\zeta_j$.
    Then there is $n_0 \geq 1$ such that for all $n \geq n_0$, we have
    $\xi^{d^n} \in B_i'$ for a unique $i$.
    Take $i$ such that $\zeta_i^{d^m} = \zeta_i$ and
    $n_1 \geq n_0$ such that $\xi^{d^{n_1}} \in B_i'$.
    Then we have $\xi^{d^{n_1 + m}} \in B_i$ and hence
    $\xi^{d^{n_1 + m}} \in B_i'$.
    Repeat this and we get $\xi^{d^{n_1 + nm}} \in B_i'$ for all $n\geq 0$.
    By \cref{eq:limP(xidn)}, we have 
    $\lim_{n \to \infty}\xi^{d^{n_1 + nm}} = \zeta_i $.
    Set $\eta = \xi^{d^{n_1}}$ and $\zeta = \zeta_i$.
    Then we have $\lim_{n \to \infty}\eta^{(d^m)^n} = \zeta$
    and $\zeta^{d^{m}}=\zeta$.
    Note that $|\zeta| = 1$ since $|\eta|=1$ and hence 
    $\zeta$ is a root of $1$.
    
    To end the proof, it is enough to show that $\eta$ is a root of $1$.
    Let us take an $n \geq 0$ such that
    $|\eta^{(d^m)^n} - \zeta| < 1$.
    Then we have
    \begin{align}
        \eta^{(d^m)^{n+1}}
        &= \left(\zeta + \eta^{(d^m)^n} - \zeta \right)^{d^m}\\
        &=\sum_{l=0}^{d^m} 
        \binom{d^m}{l} \zeta^{d^m - l} 
        \left(\eta^{(d^m)^n} - \zeta \right)^l\\
        &= \zeta^{d^m} + d^m \zeta^{d^m -1} \left(\eta^{(d^m)^n} - \zeta \right) + \sum_{l=2}^{d^m} \binom{d^m}{l} \zeta^{d^m - l} 
        \left(\eta^{(d^m)^n} - \zeta \right)^l\\
        & = \zeta + d^m \zeta^{d^m -1} \left(\eta^{(d^m)^n} - \zeta \right) + \sum_{l=2}^{d^m} \binom{d^m}{l} \zeta^{d^m - l} 
        \left(\eta^{(d^m)^n} - \zeta \right)^l.
    \end{align}
    Now since $p \not\mid d$, we have $|d| = 1$.
    Thus, combined with $|\zeta|=1$, we have
    \begin{align}
        \left| \sum_{l=2}^{d^m} \binom{d^m}{l} \zeta^{d^m - l} 
        \left(\eta^{(d^m)^n} - \zeta \right)^l \right|
        < \left| d^m \zeta^{d^m -1} \left(\eta^{(d^m)^n} - \zeta \right) \right|
    \end{align}
    for $2\leq l \leq d^m$ and hence 
    \begin{align}
        \left|\eta^{(d^m)^{n+1}} - \zeta \right| 
        = \left| d^m \zeta^{d^m -1} \left(\eta^{(d^m)^n} - \zeta \right) \right|
        = \left|\eta^{(d^m)^n} - \zeta  \right|.
    \end{align}
    This implies $\left|\eta^{(d^m)^n} - \zeta  \right|$ is constant as 
    $n$ goes to infinity. Since $\eta^{(d^m)^n}$ converges to $\zeta$, we have $\eta^{(d^m)^n} = \zeta$ for a large $n$
    and we are done.
\end{proof}

\begin{lemma}\label{lem:div_and_fld_ext}
    Let $K \subset L$ be a field extension.
    Let $f \in K\lb T_1, \dots ,T_n \rb \smallsetminus \{0\}$ and
    $g \in L\lb T_1, \dots ,T_n \rb$.
    If $fg \in K\lb T_1, \dots ,T_n \rb$, then we have $g \in K\lb T_1, \dots ,T_n \rb$.
\end{lemma}

\begin{proof}
    We prove by induction on $n$.
    If $n=0$, then the statement is trivial.
    Suppose $n \geq 1$. Write
    \begin{align}
        f = \sum_{i \geq 0} a_i T_n^i,\ g = \sum_{j \geq 0} b_j T_n^j
    \end{align}
    where $a_i \in K\lb T_1, \dots, T_{n-1} \rb, b_j \in L\lb T_1, \dots, T_{n-1} \rb$.
    As $fg \in K\lb T_1, \dots ,T_n \rb$, 
    \begin{align}
        a_0b_0,\ a_0b_1 + a_1 b_0,\dots,\sum_{i=0}^k a_ib_{k-i},\dots
    \end{align}
    are contained in $K\lb T_1, \dots, T_{n-1} \rb$.
    Using the induction hypothesis, we inductively see that 
    $b_j \in K\lb T_1, \dots, T_{n-1} \rb$ for all $j$.
\end{proof}

\bibliographystyle{abbrv}
\bibliography{arXiv_version/main_arXiv}

\end{document}